\documentclass[final,12pt,a4paper]{amsart}
    \title[Simple-minded systems and reduction]{Simple-minded systems and reduction for negative Calabi-Yau triangulated categories}

\author{Raquel Coelho Sim\~oes}
\author{David Pauksztello}

\usepackage{amssymb}
\usepackage{latexsym}
\usepackage{amsmath}
\usepackage{mathrsfs}
\usepackage[all]{xy}
\usepackage{enumitem}
\usepackage{graphicx}
\usepackage{fullpage}

\usepackage[usenames,dvipsnames]{xcolor}
\usepackage{hyperref}
\hypersetup{
    colorlinks,
    linkcolor={red!50!black},
    citecolor={blue!50!black},
    urlcolor={blue!80!black}
}

\newcommand{\harxiv}[1]{ \href{http://arxiv.org/abs/#1}{\texttt{arXiv:#1}}}
\newcommand{\hyref}[2]{ \hyperref[#2]{#1~\ref*{#2}} }

\theoremstyle{plain}
\newtheorem{theorem}{Theorem}[section]
\newtheorem{lemma}[theorem]{Lemma}
\newtheorem{corollary}[theorem]{Corollary}
\newtheorem{proposition}[theorem]{Proposition}
\newtheorem{introtheorem}{Theorem}

\newtheorem{claim}{Claim}

\theoremstyle{definition}
\newtheorem{remark}[theorem]{Remark}
\newtheorem{example}[theorem]{Example}
\newtheorem{definition}[theorem]{Definition} 
\newtheorem{setup}[theorem]{Setup}
\newtheorem{notation}[theorem]{Notation}

\newcommand{\Step}[1]{\medskip\noindent\textbf{Step #1.}}

\newcommand{\sD}{\mathsf{D}}
\newcommand{\sL}{\mathsf{L}}
\newcommand{\sR}{\mathsf{R}}
\newcommand{\sS}{\mathsf{S}}
\newcommand{\sT}{\mathsf{T}}
\newcommand{\sX}{\mathsf{X}}
\newcommand{\sY}{\mathsf{Y}}
\newcommand{\sZ}{\mathsf{Z}}

\newcommand{\Db}{\mathsf{D}^b}

\newcommand{\bZ}{\mathbb{Z}}

\renewcommand{\geq}{\geqslant}
\renewcommand{\leq}{\leqslant}
\renewcommand{\phi}{\varphi}
\renewcommand{\epsilon}{\varepsilon}

\DeclareMathOperator{\Hom}{\mathsf{Hom}}

\DeclareMathOperator{\id}{{\mathsf{id}}}

\DeclareMathOperator{\add}{\mathsf{add}}

\newcommand{\kk}{{\mathbf{k}}}

\newcommand{\SSS}{\mathbb{S}}

\newcommand{\extn}[1]{\langle #1 \rangle}
\newcommand{\extnZ}[1]{\{ #1 \}}

\newcommand{\Rmu}[1]{\sR_{\sS}(#1)}
\newcommand{\Lmu}[1]{\sL_{\sS}(#1)}

\newcommand{\orth}{^\perp}
\newcommand{\orthw}{^{\perp_w}}

\newcommand{\Zone}{\mathbf{(Z1)}}
\newcommand{\Ztwo}{\mathbf{(Z2)}}

\newcommand{\TRone}{\mathbf{(TR1)}}
\newcommand{\TRtwo}{\mathbf{(TR2)}}
\newcommand{\TRthree}{\mathbf{(TR3)}}
\newcommand{\TRfour}{\mathbf{(TR4)}}

\newcommand{\too}{\longrightarrow}
\newcommand{\rightlabel}[1]{\stackrel{#1}{\longrightarrow}}

\newcommand{\tri}[3]{#1\rightarrow #2\rightarrow #3\rightarrow \Sigma #1}
\newcommand{\triZ}[3]{#1\rightarrow #2\rightarrow #3\rightarrow #1\extn{1}}
\newcommand{\trilabel}[4]{#1\stackrel{#4}{\longrightarrow} #2\longrightarrow #3\longrightarrow \Sigma #1}
\newcommand{\trilabelsZ}[6]{#1\stackrel{#4}{\longrightarrow} #2\stackrel{#5}{\longrightarrow} #3\stackrel{#6}{\longrightarrow}  #1 \extn{1}}
\newcommand{\trilabels}[6]{#1\stackrel{#4}{\longrightarrow} #2\stackrel{#5}{\longrightarrow} #3\stackrel{#6}{\longrightarrow} \Sigma #1}

\newenvironment{pmat}{\left[ \begin{smallmatrix}}{\end{smallmatrix} \right]}

\entrymodifiers={+!!<0pt,\fontdimen22\textfont2>}

\newcommand{\coloneqq}{\mathrel{\mathop:}=}

\usepackage{tikz}
	\usetikzlibrary{decorations.pathmorphing}
        \usetikzlibrary{topaths}
        \usetikzlibrary{automata,arrows}
        \usetikzlibrary{shapes.geometric} 

\begin{document}

\begin{abstract}
We develop the basic properties of $w$-simple-minded systems in $(-w)$-Calabi-Yau triangulated categories for $w \geq 1$. The main result is a reduction technique for negative Calabi-Yau triangulated categories. We show that the theory of simple-minded systems exhibits striking parallels with that of cluster-tilting objects. Our construction provides an inductive technique for constructing simple-minded systems.
\end{abstract}

\keywords{Simple-minded system, simple-minded mutation pair, Calabi-Yau triangulated category, Calabi-Yau reduction}

\subjclass[2010]{18E30,16G10}

\maketitle

{\small
\setcounter{tocdepth}{1}
\tableofcontents
}

\vspace{-1cm}

\addtocontents{toc}{\protect{\setcounter{tocdepth}{-1}}}  
\section*{Introduction} 
\addtocontents{toc}{\protect{\setcounter{tocdepth}{1}}}   

From a representation-theoretic perspective, two algebras are equivalent when they have the same representation theory, i.e. equivalent module categories. Module categories have two important types of generators: projective modules and simple modules. 
Morita theory describes equivalences of module categories in terms of images of projective modules. Tilting theory is the generalisation of Morita theory to derived categories, describing equivalences of categories in terms tilting objects. Tilting objects, and more generally silting objects, can be thought of as `projective-minded' objects; see \cite{CKL}. 

Let $d \geq 2$. In a $d$-Calabi-Yau triangulated category the `projective-minded' objects are the $d$-cluster-tilting objects. Silting objects and $d$-cluster-tilting objects admit a mutation procedure \cite{AI, IY} enabling the construction of new silting objects or $d$-cluster-tilting objects from old ones.
Such mutation procedures were key in the categorification of Fomin and Zelevinsky's cluster algebras; see \cite{BMRRT,FZ}.
Modelled on the cluster mutation procedure, Iyama and Yoshino defined the notion of a mutation pair in a triangulated category and used this to construct a subfactor triangulated category \cite{IY}. This subfactor triangulated category has remarkable properties: it is smaller and simpler than the original category, if the original category was Calabi-Yau so is the new one, and there is a bijection between cluster-tilting objects in the subfactor and cluster tilting objects in the big category containing a given summand. This provides a powerful inductive technique -- known as \emph{Iyama-Yoshino reduction} -- for constructing cluster-tilting objects and studying mutation.
Iyama-Yoshino reduction has produced many generalisations and applications, for instance \cite{Iyama-Wemyss, Iyama-Yang, Li, Nakaoka, Wei, ZXO, ZZ}.
For example, \cite{Iyama-Yang} connects the silting reduction of \cite{AI} to Iyama-Yoshino reduction in the context of Amiot's construction of the generalised cluster category \cite{Amiot}.

The other kind of triangulated category important in representation theory is the stable module category. For such categories, an analogue of Morita theory is missing because the projective objects become invisible; the only natural generators are the simple modules. 
This has led to the study of `simple-minded' objects in \cite{Al-Nofayee, Asai, CKL, CS12, CS15, CS17, Dugas-tilting, Dugas, Koenig-Liu, Koenig-Yang}.
The central open problem for stable module categories is the Auslander--Reiten conjecture \cite{AR} which states that stable equivalences preserve the number of isomorphism classes of non-projective simple modules. This is known to hold if and only if it holds for selfinjective algebras \cite{Martinez}. It is therefore an important problem to study `simple-minded' objects in the context of selfinjective algebras. 

Simple-minded systems were introduced in \cite{Koenig-Liu} and their mutation theory was formally defined in \cite{Dugas}, cf. \cite{Koenig-Yang} in the setting of the derived category with simple-minded collections. The theory developed by Dugas in \cite{Dugas} works in the context of a $(-1)$-Calabi-Yau triangulated category. This is a natural setting: the key class of examples of selfinjective algebras are symmetric algebras, whose stable module categories are $(-1)$-Calabi-Yau.
In the context of triangulated categories generated by spherical objects, the first author extended the notion of simple-minded system to $w$-simple-minded system in $(-w)$-Calabi-Yau triangulated categories in \cite{CS17}, where $w \geq 1$. In that article, the first author observed a striking parallel between the mutation theory for $w$-simple-minded systems and $d$-cluster-tilting subcategories in the positive Calabi-Yau case tackled in \cite{HJ}.

In \cite{CS12, CS15}, closely related notions of Hom-configurations and Riedtmann configurations, inspired by Riedtmann's seminal classification of finite-type selfinjective algebras \cite{Riedtmann}, were studied in $(-1)$-Calabi-Yau orbit categories. In \cite{CS12}, a finite-type version of a simple-minded analogue of Iyama-Yoshino reduction was a crucial technical tool in building Hom-configurations and establishing a bijection with noncrossing partitions. In particular, a double-perpendicular category of a partial Hom-configuration was shown to be a $(-1)$-Calabi-Yau orbit category of smaller size. This led the first author to conjecture in \cite{CS-thesis} that this provided a `simple-minded' analogue of Iyama-Yoshino reduction.
By modelling our reduction technique on Dugas' definition of simple-minded mutation from \cite{Dugas}, we establish this conjecture. This is the first main result of this article; we refer the reader to Sections~\ref{sec:sms} and \ref{sec:mutation} for the required definitions.

\begin{introtheorem}[Theorems~\ref{thm:pretriangulated} and \ref{thm:triangulated}]
Let $\sD$ be a Hom-finite, Krull-Schmidt, $\kk$-linear triangulated category and $w\geq 1$. Suppose $\sS$ is a $w$-orthogonal collection whose extension closure $\extn{\sS}$ is functorially finite. Moreover assume that $\sS$ is an $\SSS_{-w}$-subcategory of $\sD$. Let 
\[
\sZ = \{d \in \sD \mid \Hom(\Sigma^i \sS, d) = 0 \text{ and } \Hom(d, \Sigma^{-i} \sS) = 0 \text{ for all } i=0,\ldots,w-1 \}.
\]
Then $\sZ$ is a triangulated category.
\end{introtheorem}

Note that $\sZ$ is not a triangulated subcategory $\sD$: it has a newly defined triangulated structure.
Our second main result says that this provides an inductive technique for building $w$-simple-minded systems.

\begin{introtheorem}[Theorem~\ref{thm:bijectionsms}]
Let $\sD$ be a Hom-finite, Krull-Schmidt, $\kk$-linear triangulated category and $w\geq 1$. Suppose $\sS$ is a $w$-orthogonal collection whose extension closure $\extn{\sS}$ is functorially finite and $\sZ$ be as above. Then there is a bijection,
\[
\{ \text{$w$-simple minded systems in $\sD$ containing $\sS$} \} \stackrel{1-1}{\longleftrightarrow}
\{ \text{$w$-simple minded systems in $\sZ$} \}.
\]
\end{introtheorem}

Our final main result says that the Calabi-Yau type of the category is preserved by our construction.

\begin{introtheorem}[Theorem~\ref{thm:Serre}]
Let $\sD$ be a Hom-finite, Krull-Schmidt, $\kk$-linear triangulated category and $w\geq 1$. Suppose $\sS$ is a $w$-orthogonal collection whose extension closure $\extn{\sS}$ is functorially finite and $\sZ$ be as above. If $\sD$ is $w$-Calabi-Yau then so is $\sZ$.
\end{introtheorem}

The three theorems above tell us that our version of reduction for Calabi-Yau triangulated categories, which is compatible with simple-minded systems, is a complete analogue of the theory developed for cluster-tilting in \cite{IY}. We briefly comment on the differences: our construction is not a subfactor triangulated category. Indeed in Example~\ref{subfactor} we explain why the subfactor construction does not work in our situation in a simple example. While modelling the reduction construction on the mutation theory is the natural thing to do, the arguments and proofs are very different from those in \cite{IY}. One can think of this as manifestation of the differences observed between simple-minded collections and silting objects observed in \cite{Koenig-Yang}.

Finally, we comment on the structure of the paper. In Section~\ref{sec:background} we set up our notation and recall some key results on approximation theory. Section~\ref{sec:sms} is a formal treatment of $w$-simple-minded systems, generalising \cite{CS17, Dugas}, providing their basic properties; here one will see already the striking parallel with cluster theory. Section~\ref{sec:mutation} sets up the notion of simple-minded mutation pair which will be used for the reduction. Section~\ref{sec:pretriangulated} shows that the reduced category is pretriangulated, the proof of the Octahedral Axiom is quite involved so is deferred to Section~\ref{sec:triangulated}. Finally in Section~\ref{sec:cy-reduction} we show that this provides an inductive technique for constructing simple-minded systems. Section~\ref{sec:examples} illustrates the theory with some simple examples.

\section{Background} \label{sec:background}

Throughout this paper, $\kk$ denotes an algebraically closed field, and $\sD$ denotes a Hom-finite, $\kk$-linear, Krull-Schmidt triangulated category with shift functor $\Sigma \colon \sD \to \sD$. We shall assume all subcategories are full and strict.

\subsection{Pretriangulated versus triangulated categories}
Basic properties of triangulated categories can be found in \cite{Happel,Hartshorne,HJ-triang,Neeman-book}. In this paper we enumerate the axioms of triangulated categories as in \cite[Definition I.1.1]{Happel}, i.e. the four axioms will be denoted $\TRone, \TRtwo$, $\TRthree$ and $\TRfour$. 
If $(\sD,\Sigma)$ satisfy only axioms $\TRone, \TRtwo$ and $\TRthree$ then $\sD$ is called a \emph{pretriangulated category}.
The axiom $\TRfour$ is often called the \emph{Octahedral Axiom}, since there are many equivalent formulations of this axiom, see for example \cite{Hubery,Neeman}, we explicitly state the formulation we shall use in this article below. 
\begin{itemize}
\item[$\TRfour$]
Given two triangles $\trilabels{x}{y}{z}{f}{g}{h}$ and $\trilabels{y}{y'}{y''}{u}{v}{w}$ then there is a commutative diagram
\[
\xymatrix@!R=8px{
                                  & \Sigma^{-1} y'' \ar@{=}[r] \ar[d]^{-\Sigma^{-1} w} & \Sigma^{-1} y'' \ar[d]^{-\Sigma^{-1} w'}   &    \\
x \ar[r]^{f} \ar@{=}[d] & y \ar[r]^g \ar[d]^u                                                  & z \ar[r]^h \ar[d]^{u'}                                 & \Sigma x \ar@{=}[d] \\
x \ar[r]_{uf}                & y' \ar[r]_{g'} \ar[d]^v                                              & z' \ar[r]_{h'} \ar[d]^{v'}                              & \Sigma x \\
                                 & y'' \ar@{=}[r]                                                          & y''                                                            &
}
\]
in which each row and column is a triangle and $(\Sigma f) h' = wv'$.
\end{itemize}

For subcategories $\sX, \sY$ of a (pre)triangulated category $\sD$, we define 
\[
\sX * \sY = \{ d \in \sD \mid  \text{there exists a triangle } \tri{x}{d}{y} \text{ with } x \in \sX \text{ and } y \in \sY \}.
\]
 The subcategory $\sX$ is said to be {\it extension-closed} if $\sX * \sX = \sX$. We denote by $\extn{\sX}$ the {\it extension closure} of $\sX$, i.e. the smallest subcategory of $\sD$ which contains $\sX$ and is extension-closed.
 We define the \emph{right and left perpendicular categories} as follows:
 \[
 \sX\orth = \{ d \in \sD \mid \Hom(\sX,d) = 0 \}
 \text{ and }
 {}\orth \sX = \{ d \in \sD \mid \Hom(d,\sX) = 0 \},
 \]
 where $\Hom(\sX,d) = 0$ is shorthand for $\Hom(x,d) = 0$ for each object $x$ of $\sX$; likewise for $\Hom(d,\sX)$.
 
 Finally, a \emph{Serre functor} on $\sD$ is an autoequivalence $\SSS \colon \sD \to \sD$ such that there is an isomorphism,
 \[
 \Hom(x,y) \simeq D\Hom(y, \SSS x),
 \]
which is natural in $x$ and $y$, where $D$ denotes the standard vector space duality. If $\sD$ has a Serre functor it is unique up to isomorphism and $\sD$ is said to satisfy \emph{Serre duality}. For details we refer to \cite{RvdB}.

Let $w\in \bZ$. A triangulated category $\sD$ satisfying Serre duality is said to be \emph{$w$-Calabi-Yau} (\emph{$w$-CY} for short) if there is a natural isomorphism $\SSS \simeq \Sigma^w$, where $\SSS$ is the Serre functor on $\sD$. 

\subsection{Approximation theory}  
Let $\sX$ be a subcategory of $\sD$, and $d$ an object in $\sD$. A morphism $f \colon x \to d$, with $x \in \sX$, is said to be: 
\begin{enumerate}
\item a {\it right $\sX$-approximation of $d$} if $\Hom(\sX, f) \colon \Hom(\sX,x) \to \Hom(\sX,d)$ is surjective; 
\item {\it right minimal} if any $g \colon x \to x$ satisfying $f g = f$ is an automorphism;
\item a {\it minimal right $\sX$-approximation of $d$} if it is both a right $\sX$-approximation of $x$ and right minimal. 
\end{enumerate}
If every object in $\sD$ admits a right $\sX$-approximation, then $\sX$ is said to be \emph{contravariantly finite}. There are dual notions of \emph{(minimal) left $\sX$-approximations} and \emph{covariantly finite} subcategories. The subcategory $\sX$ of $\sD$ is called \emph{functorially finite} if it is both contravariantly finite and covariantly finite. 

We now collect some basic properties of approximations which will be used throughout the paper. 

\begin{lemma} \label{lem:basic-props} 
Let $\sX \subseteq \sD$ be an extension-closed subcategory and let $d \in \sD$. 
\begin{enumerate}
\item Suppose $d$ admits a right $\sX$-approximation. Then $d$ admits a minimal right $\sX$-approximation, which is unique up to isomorphism.
\item If $\alpha \colon x \rightarrow d$ is a minimal right $\sX$-approximation, then each right $\sX$-approximation of $d$ is, up to isomorphism, of the form $\begin{pmat} \alpha & 0 \end{pmat} \colon x \oplus x^\prime \rightarrow d$.
\item If $\alpha \colon x \rightarrow d$ is a minimal right $\sX$-approximation and $\beta \colon x^\prime \rightarrow d$ is a right $\sX$-approximation, then $x$ is a summand of $x^\prime$.
\item If $\alpha \colon x \rightarrow d$ is a right $\sX$-approximation and $\trilabels{x}{d}{y}{\alpha}{\beta}{}$ is its completion to a distinguished triangle, then $y \in \sX\orth$ and $\beta \colon d \rightarrow y$ is a left $(\sX \orth)$-approximation of $d$.
\item If $\beta \colon d \rightarrow x$ is a left $\sX$-approximation and $\trilabels{z}{d}{x}{\alpha}{\beta}{}$ is its completion to a distinguished triangle, then $z \in {}\orth \sX$ and $\alpha \colon z \rightarrow d$ is a right $({}\orth \sX)$-approximation of $d$.
\end{enumerate}
\end{lemma}

\begin{proof}
The first three statements are well-known and straightforward, see for example \cite{AS}. The final two are known as the Wakamatsu lemma for triangulated categories; see \cite[Section 2]{IY} or \cite[Lemma 2.1]{J}, for example.
\end{proof}

\section{Simple-minded systems} \label{sec:sms}

Simple-minded systems were introduced in \cite{Koenig-Liu} (see also \cite{Po}) and generalised to $w$-simple-minded systems for $w \geq 1$ in \cite{CS17}.
We start by reviewing the definitions and basic properties of these concepts.

Given a collection of objects $\sX$ in $\sD$, we denote by $\sX^\oplus$ the smallest subcategory of $\sD$ containing $\sX$ and closed under direct summands. We will also use the following notation: $(\sX)_1 \coloneqq \sX$ and $(\sX)_n \coloneqq (\sX * (\sX)_{n-1})^\oplus$, for $n \geq 2$. 

\begin{definition}\label{def:w-sms}
Let $w \geq 1$. A collection of objects $\sS$ in $\sD$ is \textit{$w$-orthogonal} if
\begin{enumerate}
\item $\dim \Hom_\sD(x,y) = \delta_{xy}$, for every $x, y \in \sS$;
\item If $w \geq 2$, $\Hom(\Sigma^k x,y) = 0$, for $1 \leq k \leq w-1$ and $x, y \in \sS$;
\end{enumerate}
A $w$-orthogonal collection $\sS$ is a \emph{$w$-simple minded system} if additionally,
\begin{enumerate}[resume]
\item $\sD = \add(\extn{\sS, \Sigma^{-1} \sS, \ldots,\Sigma^{1-w} \sS})$.
\end{enumerate}
\end{definition}

We recall the following definition from \cite{CS15}, which was inspired by \cite{Riedtmann}.

\begin{definition}
Let $w \geq 1$. A $w$-orthogonal collection $\sS$ is called a \emph{left (resp. right) $w$-Riedtmann configuration} if for each $d$ in $\sD$ with $\Hom(\Sigma^k s, d) = 0$ (resp. $\Hom(\Sigma^k d, s) = 0$) for each $s \in \sS$ and $0 \leq k \leq w-1$ then $d = 0$. 
A $w$-orthogonal collection $\sS$ is a \emph{$w$-Riedtmann configuration} if it is both a left and right $w$-Riedtmann configuration.
\end{definition}

A $1$-orthogonal collection of objects will be referred to as simply an \emph{orthogonal collection}. An orthogonal collection is called a \emph{system of orthogonal bricks} in \cite{Po}, a \emph{set of (pairwise) orthogonal bricks} in \cite{Dugas} and a \emph{semibrick} in \cite{Asai}.  
We now recall some basic properties of orthogonal collections from \cite{Dugas}.

\begin{lemma} \label{lem:Dugas}
Let $\sS$ be an orthogonal collection in $\sD$. Then the following assertions hold.
\begin{enumerate}
\item (\cite[Lemma 2.3]{Dugas}) $\extn{\sS} = \bigcup_{n \geq 1} (\sS)_n$.
\item (\cite[dual of Lemma 2.6]{Dugas})  If there is a triangle $\trilabel{s}{x}{y}{\sigma}$ with $s \in \sS$, $x \in (\sS)_n$ and $\sigma \neq 0$, then $y \in (\sS)_{n-1}$.
\item (\cite[Lemma 2.7]{Dugas}) $(\sS)_n$ is closed under direct summands for each $n \geq 1$. In particular $\extn{\sS}$ is closed under direct summands.
\end{enumerate} 
\end{lemma}

In light of Lemma~\ref{lem:Dugas} the following definition makes sense.

\begin{definition}[{\cite[Def. 2.5]{Dugas}}] \label{def:length}
Let $\sS$ be an orthogonal collection in $\sD$. Let $x \in \extn{\sS}$. The \emph{$\sS$-length} of $x$ is the minimum natural number $n$ such that $x \in (\sS)_n$.
In particular, this means that there is an \emph{$\sS$-composition series},
\begin{align*}
& \tri{s_1}{x_2}{s_2} \\
& \tri{x_2}{x_3}{s_3} \\
& \hspace{1.75cm} \vdots  \\
& \tri{x_{n-1}}{x}{s_n},
\end{align*}
with $s_i \in \sS$ and $x_i \in (\sS)_i$ for $i = 2,\ldots,n-1$.
\end{definition}

We recall the following definition given in \cite{IY}.

\begin{definition} \label{def:Sw-subcat}
Let $w \in \bZ$, and assume $\sD$ has a Serre duality $\SSS$. Write $\SSS_w = \SSS \Sigma^{-w}$. A subcategory $\sX$ of $\sD$ is said to be an {\it $\SSS_{w}$-subcategory} of $\sD$ if $\sX =  \SSS_w \sX = \SSS^{-1}_w \sX$. 
\end{definition}

Suppose that $\sS$ is an orthogonal collection in $\sD$ such that $\extn{\sS}$ is functorially finite. In particular, by \cite[Proposition 2.3]{IY}, there are torsion pairs $(\extn{\sS},\sS\orth)$ and $({}\orth \sS, \extn{\sS})$ in $\sD$. The following lemma is a useful generalisation of \cite[Lemma 4.7]{Dugas}; there is also a corresponding generalisation of its dual \cite[Lemma 4.6]{Dugas}.

\begin{lemma} \label{lem:torsion-pair}
Let $d \in \sD$. Suppose $\sS$ is an orthogonal collection in $\sD$ such that $\extn{\sS}$ is functorially finite. Let $\trilabels{s_d}{d}{z_d}{f}{g}{}$ be a decomposition triangle with respect to the torsion pair $(\extn{\sS},\sS\orth)$, in which $f$ and $g$ are minimal right and left approximations, respectively (cf. \cite[Lemma 3.2]{Dugas}).
\begin{enumerate}
\item The map $\Hom(\sS,f) \colon \Hom(\sS, s_d) \to \Hom(\sS, d)$ is an isomorphism.
\end{enumerate}
Suppose further that $w \geq 1$ and $\sS$ is an $\SSS_{-w}$-subcategory of $\sD$. 
\begin{enumerate}[resume]
\item The map $\Hom(\Sigma^{w-1}g, \sS) \colon \Hom(\Sigma^{w-1} z_d, \sS) \to \Hom(\Sigma^{w-1} d, \sS)$ is a monomorphism if and only if $\Hom(\sS,f)$ is an isomorphism. 
\item If $d \in {}\orth (\Sigma^{1-w} \sS)$ then $z_d \in {}\orth (\Sigma^{1-w} \sS)$.
\end{enumerate}
\end{lemma}

\begin{proof}
The first statement is \cite[Lemma 4.7(a)]{Dugas}; see \cite[Lemma 4.6(a)]{Dugas} for a proof in the dual case. Note that the proof in \cite[Lemma 4.6(a)]{Dugas} requires only that $\sS$ is an orthogonal collection such that $\extn{\sS}$ is functorially finite: the other blanket assumptions in that section are not used in the argument.

The second statement is essentially \cite[Lemma 4.7(b)]{Dugas}. However, since the argument in \cite{Dugas} is formulated for the case when $w = 1$, we give a brief sketch of the adaptations. First, applying the functor $\Hom(\sS,-)$ to the decomposition triangle gives the long exact sequence, where $(\sS,f) = \Hom(\sS,f)$,
\[
\xymatrix{
\Hom(\Sigma \sS, d) \ar[r]^-{(\Sigma \sS, g)} & \Hom(\Sigma \sS, z_d) \ar[r] & \Hom(\sS, s_d) \ar@{->>}[r]^-{(\sS,f)} & \Hom(\sS, d).
}
\]
We know that $\Hom(\sS,f)$ is surjective because $f \colon s_d \to d$ is a right $\extn{\sS}$-approximation. Therefore, $\Hom(\Sigma \sS,g)$ is surjective if and only if $\Hom(\sS,f)$ is an isomorphism.

We claim that $\Hom(g, \Sigma^{-w+1} \sS)$ is injective if and only if $\Hom(\Sigma \sS,g)$ is surjective.
Note that $\Hom(g, \Sigma^{-w+1} \sS)$ is injective if and only if $D \Hom(g, \Sigma^{-w+1} \sS)$ is surjective. Using the fact that $\sS$ is an $\SSS_{-w}$-subcategory, Serre duality gives the following commutative diagram, which establishes the claim.
\[
\xymatrix{
D \Hom(d,\Sigma^{-w+1} \sS) \ar[d]^-{\sim} \ar[rrr]^-{D \Hom(g,\Sigma^{-w+1} \sS)} & & & D \Hom(z_d, \Sigma^{-w+1} \sS) \ar[d]^-{\sim} \\
\Hom(\Sigma \sS, d) \ar[rrr]_-{\Hom(\Sigma \sS, g)}                                                   & & & \Hom(\Sigma \sS, z_d) 
} 
\]
The final statement is immediate from the second statement.
\end{proof}

The following observation will be useful later.

\begin{lemma}\label{lem:reverseorderextension}
Let $w \geq 2$. If $\sS$ is a $w$-orthogonal collection then $\extn{\sS} * \Sigma^i \extn{\sS} \subseteq \Sigma^i \extn{\sS} * \extn{\sS}$ for $0 < i < w$. 
\end{lemma}

In general the inclusion of Lemma~\ref{lem:reverseorderextension} is strict. 

\begin{proof} 
By Lemma~\ref{lem:Dugas}, $\extn{\sS} = \bigcup_{n \geq 0} (\sS)_n$. Let $0 < i < w$. First, we prove, by induction on $n$, that $\extn{\sS} * \Sigma^i \sS = \Sigma^i \sS * \extn{\sS}$. Let $d \in \sS * \Sigma^i \sS$. We have a triangle $\trilabels{s'}{d}{\Sigma^i s''}{}{}{f}$, with $s', s'' \in \sS$. Since $\sS$ is a $w$-orthogonal collection, we have $f = 0$ or $f$ is an isomorphism. Therefore, $d \simeq s' \oplus \Sigma^i s''$ or $d=0$, and in both cases, we have $d \in \Sigma^i \sS * \sS$. 

Assume, by induction hypothesis, that $(\sS)_n * \Sigma^i \sS \subseteq \Sigma^i \sS * \extn{\sS}$. Let $d \in (\sS)_{n+1} * \Sigma^i \sS$, and write $\tri{s'}{d}{\Sigma^i s''}$, with $s' \in (\sS)_{n+1}$ and $s^{''} \in \sS$. Since $(\sS)_{n+1} = (\sS)_n * \sS$, we have a triangle of the form $\tri{x}{s'}{s}$, with $x \in (\sS)_n$ and $s \in \sS$. By the Octahedral Axiom in $\sD$, we have the following diagram.
\[
\xymatrix@!R=8px{
                    & \Sigma^{i-1} s^{''} \ar@{=}[r] \ar[d] & \Sigma^{i-1} s^{''} \ar[d]^f  & \\
x \ar[r] \ar@{=}[d] & s' \ar[r] \ar[d]                    & s \ar[r] \ar[d]             & \Sigma x \ar@{=}[d] \\
x \ar[r]            & d \ar[r] \ar[d]                     & y \ar[r] \ar[d]             & \Sigma x \\
                    & \Sigma^i s^{''} \ar@{=}[r]           & \Sigma^i s^{''}              &
}
\]
Since $\sS$ is a $w$-orthogonal collection, it follows that $f$ is an isomorphism or $f= 0$. Hence,  either $y \simeq 0$, which implies $d \simeq x \in (\sS)_n \subseteq \Sigma^i \sS * \extn{\sS}$, or $y \simeq s \oplus \Sigma^i s^{''}$. In this case, we have the following octahedral diagam.
\[
\xymatrix@!R=8px{
                    & \Sigma^i s^{''} \ar@{=}[r] \ar[d]       & \Sigma^i s^{''} \ar[d]  & \\
d \ar[r] \ar@{=}[d] & s \oplus \Sigma^i s^{''} \ar[r] \ar[d]  & \Sigma x \ar[r] \ar[d] & \Sigma d \ar@{=}[d] \\
d \ar[r]            & s \ar[r] \ar[d]                        & \Sigma z \ar[r] \ar[d] & \Sigma d \\
                    & \Sigma^{i+1} s^{''} \ar@{=}[r]           & \Sigma^{i+1} s^{''}      & 
}
\]
From the right-hand vertical triangle, we have $z \in (\sS)_n * \Sigma^i \sS$. Hence, by the induction hypothesis, $z \in \Sigma^i \sS * \extn{\sS}$.
Thus $d \in \Sigma^i \sS * \extn{\sS} * \sS = \Sigma^i \sS * \extn{\sS}$.
This finishes the proof that $\extn{\sS} * \Sigma^i \sS \subseteq \Sigma^i \sS * \extn{\sS}$. 

Now assume by induction that $\extn{\sS} * \Sigma^i (\sS)_n \subseteq \Sigma^i \extn{\sS} * \extn{\sS}$. We want to prove that $\extn{\sS} * \Sigma^i (\sS)_{n+1} \subseteq \Sigma^i \extn{\sS} * \extn{\sS}$. Consider the triangle $\tri{s'}{d}{\Sigma^i s^{''}}$, with $s' \in \extn{\sS}, s^{''} \in (\sS)_{n+1}$. Then there is a triangle $\tri{x}{s^{''}}{s_1}$, with $x \in (\sS)_n$ and $s_1 \in \sS$, and by the Octahedral Axiom, we have the commutative diagram below.
\[
\xymatrix@!R=8px{
                              & d \ar@{=}[r] \ar[d]           & d \ar[d]                   & \\
\Sigma^i x \ar[r] \ar@{=}[d]  & \Sigma^i s^{''} \ar[r] \ar[d]  & \Sigma^i s_1 \ar[r] \ar[d]  & \Sigma^{i+1} x \ar@{=}[d] \\
\Sigma^i x \ar[r]             & \Sigma s' \ar[r] \ar[d]       & \Sigma y \ar[r] \ar[d]      & \Sigma^{i+1} x \\
                              & \Sigma d \ar@{=}[r]           &  \Sigma d                   &  
}
\]
Hence, from the lower horizontal triangle we obtain $y \in \extn{\sS} * \Sigma^i (\sS)_n$. Hence, by induction, $y \in \Sigma^i \extn{\sS} * \extn{\sS}$.
Thus, from the right-hand column we have $d \in \Sigma^i \extn{\sS} * \extn{\sS} * \Sigma^i \sS \subseteq \Sigma^i \extn{\sS} * \Sigma^i \sS * \extn{\sS} = \Sigma^i \extn{\sS} * \extn{\sS}$, as required.
\end{proof}

\begin{lemma}
Let $w \geq 1$ and $\sS$ be a $w$-orthogonal collection. Then $\extn{\sS, \Sigma^{-1} \sS, \ldots, \Sigma^{1-w} \sS} = \extn{\sS} * \Sigma^{-1} \extn{\sS} * \cdots * \Sigma^{1-w} \extn{\sS}$.
Moreover, $\add(\extn{\sS, \Sigma^{-1} \sS, \ldots, \Sigma^{1-w} \sS}) = \extn{\sS, \Sigma^{-1} \sS, \ldots, \Sigma^{1-w} \sS}$.   
\end{lemma}

\begin{proof}
The inclusion $\extn{\sS} * \Sigma^{-1} \extn{\sS} * \cdots * \Sigma^{1-w} \extn{\sS} \subseteq \extn{\sS,\Sigma^{-1} \sS, \ldots, \Sigma^{1-w} \sS}$ is clear. For the other inclusion,
let $d \in \extn{\sS, \Sigma^{-1} \sS, \ldots, \Sigma^{1-w} \sS}$. This means there is a tower
\[
\xymatrix@!R=8px{
0 = d_0 \ar[r] & d_1 \ar[r] \ar[d] & d_2 \ar[r] \ar[d] & \cdots \ar[r] & d_{n-1} \ar[r] \ar[d] & d_n = d\ar[d] \\
&                                \Sigma^{i_1} s_{j_1} \ar@{~>}[ul] & \Sigma^{i_2} s_{j_2} \ar@{~>}[ul] & \cdots & \Sigma^{i_n-1} s_{j_n-1} \ar@{~>}[ul] & \Sigma^{i_n} s_{j_n} \ar@{~>}[ul]
}
\]
where $i_1, \ldots, i_n \in \{1-w, \ldots, 0\}$, and $s_{j_k} \in \sS$, with $1 \leq k \leq n$. In other words, we have $d \in \Sigma^{i_1} \sS * \Sigma^{i_2} \sS * \cdots * \Sigma^{i_n} \sS \subseteq \Sigma^{i_1} \extn{\sS} * \Sigma^{i_2} \extn{\sS} * \cdots * \Sigma^{i_n} \extn{\sS}$. By Lemma \ref{lem:reverseorderextension}, we can re-order the $i_k$ so that $0\geq i_1 \geq i_2 \geq \cdots \geq i_n \geq 1-w$, which implies that $d \in \extn{\sS} * \Sigma^{-1} \extn{\sS} * \cdots * \Sigma^{1-w} \extn{\sS}$.
The second statement follows immediately by \cite[Proposition 2.1]{IY} using the fact that $\extn{\sS}$ is closed under summands by Lemma~\ref{lem:Dugas}.
 \end{proof}

\begin{corollary}\label{cor:smsimpliesff}
If $\sS$ is a $w$-simple-minded system in $\sD$ then $\extn{\sS}$ is functorially finite in $\sD$.
\end{corollary}

\begin{lemma} \label{lem:funct-finite}
Let $\sS$ be a $w$-orthogonal collection in $\sD$ such that $\extn{\sS}$ is functorially finite in $\sD$. Then for $0 \leq k \leq w$, $\extn{\sS} * \Sigma^{-1} \extn{\sS} * \cdots * \Sigma^{-k} \extn{\sS}$ is functorially finite in $\sD$.
\end{lemma}

\begin{proof}
We proceed by induction on $k$. For $k = 0$ this is by assumption. Let $d \in \sD$ and fix $d_0 = d$. Suppose, by induction, for $0 \leq i < k$ we have constructed triangles
\[
x_i \rightlabel{f_i} d \rightlabel{g_i} d_{i+1} \too \Sigma x_i
\]
in which $f_i \colon x_i \to d$ is a right $(\extn{\sS} * \Sigma^{-1} \extn{\sS} * \cdots * \Sigma^{-i} \extn{\sS})$-approximation.
Note that by Lemma~\ref{lem:basic-props}(4) we have $d_{i+1} \in (\extn{\sS} * \Sigma^{-1} \extn{\sS} * \cdots * \Sigma^{-i} \extn{\sS})\orth$.
Now take a right $(\Sigma^{-k} \extn{\sS})$-approximation of $d_k$,
\[
\Sigma^{-k} s_k \too d_k \rightlabel{h_k} d_{k+1} \too \Sigma^{-k+1} s_k.
\]
Applying the Octahedral Axiom to the composition $h_k g_{k-1}$ we get the following commutative diagram.
\[
\xymatrix@!R=8px{
                                              & \Sigma^{-k} s_k \ar[d] \ar@{=}[r]    & \Sigma^{-k} s_k \ar[d]                           &  \\
d \ar@{=}[d] \ar[r]^-{g_{k-1}} & d_k \ar[d]^-{h_k} \ar[r]                     & \Sigma x_{k-1} \ar[d] \ar[r]                    & \Sigma d \ar@{=}[d] \\
d \ar[r]                                   & d_{k+1} \ar[d] \ar[r]                         & \Sigma x_k \ar[d] \ar[r]_-{-\Sigma f_k}  & \Sigma d \\
                                             & \Sigma^{-k+1} s_k \ar@{=}[r]          & \Sigma^{-k+1} s_k                                &
}
\] 
By construction, we have $x_k \in \extn{\sS} * \Sigma^{-1} \extn{\sS} * \cdots * \Sigma^{-k} \extn{\sS}$.
We claim that $d_{k+1} \in (\extn{\sS} * \Sigma^{-1} \extn{\sS} * \cdots * \Sigma^{-k} \extn{\sS})\orth$. By Lemma~\ref{lem:basic-props}(4), we have $d_{k+1} \in (\Sigma^{-k} \extn{\sS})\orth$. Consider the long exact sequence for $0 \leq i \leq k$:
\[
(\Sigma^{-k+i+1} \sS, d_k) \too (\Sigma^{-k+i+1} \sS, d_{k+1}) \too (\Sigma^{-k+i} \sS, \Sigma^{-k} s_k) \too (\Sigma^{-k+i} \sS, d_k).
\] 
When $0 < i \leq k$, the first and third terms are zero by induction and $w$-orthogonality of $\sS$.
When $i = 0$, the first term is zero and the morphism $\Hom(\Sigma^{-k} \sS, \Sigma^{-k} s_k) \too \Hom(\Sigma^{-k} \sS, d_k)$ is an isomorphism by Lemma~\ref{lem:torsion-pair}(1) applied to the orthogonal collection $\Sigma^{-k} \sS$. 
This gives the claim and shows that $f_k \colon x_k \to d$ is a right $(\extn{\sS} * \Sigma^{-1} \extn{\sS} * \cdots * \Sigma^{-k} \extn{\sS})$-approximation.
Covariant finiteness is proved analogously.
\end{proof}

The following theorem is essentially \cite[Theorem 3.3]{Dugas}. For our purposes we require some information that was implicit in the proof in \cite{Dugas} but not its statement. For the convenience of the reader we include a proof to make these details explicit.

\begin{theorem}[{\cite[Theorem 3.3]{Dugas}}] \label{thm:funct-finite}
Suppose $\sS \subseteq \sT$ for an orthogonal collection $\sT$ in $\sD$. Then $\extn{\sS}$ is functorially finite in $\extn{\sT}$ and for each $x \in (\sT)_n$, with $n$ minimally chosen, there is a right $\extn{\sS}$-approximation triangle
\[
s_x \to x \to t_x \to \Sigma s_x
\]
with $t_x \in (\sT)_m \cap \sS\orth$ for some $m \leq n$ with equality if and only if $t_x \simeq x$.
\end{theorem}

\begin{corollary}
If $\sS$ is a $w$-orthogonal collection such that $\sS \subseteq \sT$ for some $w$-simple-minded system $\sT$, then $\extn{\sS}$ is functorially finite in $\sD$.
\end{corollary}

There is a dual of Theorem~\ref{thm:funct-finite} using left $\extn{\sS}$-approximation triangles.

\begin{proof}[Proof of Theorem~\ref{thm:funct-finite}]
Let $x \in (\sT)_n$ with $n$ chosen minimally. We proceed on induction on $n$. For $n =1$, one of the triangles $s_x \to x \to 0 \to \Sigma s_x$ or  $0 \to x \to t_x \to 0$ is the required triangle and the statement holds. Suppose the statement is true for objects in $(\sT)_{n-1}$.
If $x \in \sS\orth$, then there is nothing to show, so suppose $x \notin \sS\orth$. By Lemma~\ref{lem:Dugas}(2), there exists a triangle $s \rightlabel{\sigma} x \too y \too \Sigma s$ with $s\in \sS$, $\sigma \neq 0$, $y \in (\sT)_{n-1}$.
By induction, there is a right $\extn{\sS}$-approximation triangle for $y$: $s_y \to y \to t_y \to \Sigma s_y$ with $s_y \in \extn{\sS}$ and  $t_y \in (\sT)_m \cap \sS\orth$ with $m \leq n-1$ and equality if and only if $y \simeq t_y$.
Applying the Octahedral Axiom we get the following commutative diagram.
\[
\xymatrix@!R=8px{
                             & s_y \ar[d] \ar@{=}[r]    & s_y \ar[d]                      &                                 \\
x \ar@{=}[d] \ar[r] & y \ar[d] \ar[r]                 & \Sigma s \ar[d] \ar[r]      & \Sigma x \ar@{=}[d] \\
x \ar[r]                  & t_y \ar[d] \ar[r]              & \Sigma s_x \ar[d] \ar[r]  & \Sigma x   \\
                             & \Sigma s_y \ar@{=}[r] & \Sigma s_y                    &
}
\]
Clearly, $s_x \in \extn{\sS}$ and $t_y \in (\sT)_m \cap \sS\orth$ with $m \leq n-1 < n$, giving the desired right approximation triangle.
\end{proof}

The next proposition says that a $w$-simple-minded system is precisely a $w$-Riedtmann configuration in which $\extn{\sS}$ is functorially-finite, generalising \cite[Theorem 3.8]{CS17}. This supports the view advanced in \cite{CS17} that $w$-Riedtmann configurations are a negative CY analogue of weakly cluster-tilting subcategories whilst $w$-simple-minded systems correspond to cluster-tilting subcategories; cf. \cite{HJ}. 

\begin{proposition} \label{prop:Riedtmann}
Let $\sS$ be a collection of indecomposable objects in $\sD$, and $w \geq 1$. The following conditions are equivalent:
\begin{enumerate}
\item $\sS$ is a $w$-simple-minded system.
\item $\sS$ is a right $w$-Riedtmann configuration and $\extn{\sS}$ is covariantly finite. 
\item $\sS$ is a left $w$-Riedtmann configuration and $\extn{\sS}$ is contravariantly finite. 
\end{enumerate}
\end{proposition}
\begin{proof}
The implications $(1) \implies (2)$ and $(1) \implies (3)$ follow from Corollary~\ref{cor:smsimpliesff} and \cite[Lemma 2.1]{CS17}. 

For the implication $(2) \implies (1)$, let $\sS$ be a right $w$-Riedtmann configuration and assume $\extn{\sS}$ is covariantly finite. By definition, $\sS$ is $w$-orthogonal. Let $0 \neq d \in \sD$. We want to prove that $d \in \extn{\sS} * \Sigma^{-1} \extn{\sS} * \cdots * \Sigma^{1-w} \extn{\sS}$.  

Since $\sS$ is a right $w$-Riedtmann configuration, we can take $k$ maximal with $0 \leq k \leq w-1$ such that $\Hom (\Sigma^k d, s) \neq 0$ for some $s \in \sS$. 
Since $\extn{\sS}$ is covariantly finite, we can consider the triangle occurring from a minimal left $\extn{\sS}$-approximation of $\Sigma^k d$: 
\[
\Sigma^{-1} s_k \to \Sigma^k d_k \to \Sigma^k d \to s_k.
\]
We claim that $\Hom(\Sigma^i d_k, \sS) = 0$ for $k \leq i \leq w-1$. By Wakamatsu's Lemma~\ref{lem:basic-props}, $\Hom(\Sigma^k d_k, \sS) = 0$.
By $w$-orthogonality of $\sS$ and maximality of $k$, we have $\Hom (\Sigma^i d_k, \sS) = 0$, for $k+2 \leq i \leq w-1$.  Finally, we have a short exact sequence:
\[
\xymatrix{
0 \ar[r] & \Hom (\Sigma^{k+1} d_k, \sS) \ar[r]^-{f} & \Hom (s_k, \sS) \ar[r]^-{g} & \Hom (\Sigma^k d,\sS) \ar[r] & 0,
}
\]
where $g$ is an isomorphism by the dual of Lemma~\ref{lem:torsion-pair}(1). Hence, since $f$ is a monomorphism and a zero map, it follows that $\Hom (\Sigma^{k+1} d_k, \sS) = 0$.

Finally, we show that if $d \in \sD$ satisfies $\Hom(\Sigma^i d, \sS) = 0$ for $k \leq i \leq w-1$, then $d \in \extn{\sS} * \cdots * \Sigma^{1-k}\extn{\sS}$. For $k = 1$, by the above we have $\Hom(\Sigma^i d_0, \sS) = 0$ for $0 \leq i \leq w-1$. Since $\sS$ is right $w$-Riedtmann, this implies that $d_0 = 0$, in which case $d \simeq s_0 \in \extn{\sS}$. For $1 < k \leq w-1$, by induction we have $d_k \in \extn{\sS} * \cdots * \Sigma^{1-k}\extn{\sS}$, whence it follows that $d \in \extn{\sS} * \cdots * \Sigma^{-k}\extn{\sS}$, completing the proof of the implication $(2) \implies (1)$.

The proof of $(3) \implies (1)$ is similar. Indeed, it is enough to prove that $d \in \Sigma^{w-1} \extn{\sS} * \Sigma^{w-2} \extn{\sS} * \cdots * \extn{\sS}$, whose proof is dual to the one above. 
\end{proof}

\section{Simple-minded mutation pairs} \label{sec:mutation}

In this section we introduce simple-minded mutation pairs as an analogue of the mutation pairs studied in \cite{IY}. The definition of simple-minded mutation goes back to \cite{Dugas} and \cite{Koenig-Yang}. In this article we employ the conventions of \cite{Dugas}.
To start, $\sS$ will simply be a collection of objects of $\sD$.

\begin{definition}[{\cite[Definition 4.1]{Dugas}}]
Let $\sS$ be a collection of objects in $\sD$ such that $\extn{\sS}$ is functorially finite in $\sD$. Let $d$ be an object of $\sD$.
\begin{enumerate}
\item The \textit{right mutation of $d$, $\Rmu{d}$, with respect to $\sS$} is obtained from the triangle:
\[
\Sigma^{-1} \Rmu{d} \too s_d \rightlabel{\alpha_d} \Sigma d \too \Rmu{d},
\]
where $\alpha_d \colon s_d \to \Sigma d$ is a minimal right $\extn{\sS}$-approximation.
\item The \textit{left mutation of $d$, $\Lmu{d}$, with respect to $\sS$} is obtained from the triangle:
\[
\Lmu{d} \too \Sigma^{-1} d \rightlabel{\alpha^d} s^d \too \Sigma \Lmu{d},
\]
where $\alpha^d \colon \Sigma^{-1} d \to s^d$ is a minimal left $\extn{\sS}$-approximation.
\end{enumerate}
For a subcategory $\sX \subseteq \sD$ we write $\Rmu{\sX} \coloneqq \{\Rmu{x} \mid x \in \sX\}$; analogously for $\Lmu{\sX}$.
\end{definition}

The following is an analogue of \cite[Definition 2.5]{IY}. Note that there are subtle differences in the setup, for example, in \cite{IY}, the analogue of $\sS$ is required to sit inside each part of the mutation pair.

\begin{definition}
Let $\sS$ be a full subcategory of $\sD$. A pair $(\sX,\sY)$ of subcategories of $\sD$ is an \emph{$\sS$-mutation pair} if  
$\sX = {}\orth \sS \orth \cap {}\orth(\Sigma^{-1} \sS) \cap \Sigma^{-1} (\extn{\sS} * \sY)$ and $\sY = {}\orth \sS \orth \cap (\Sigma \sS)\orth \cap \Sigma (\sX * \extn{\sS})$.
\end{definition}

The following lemma is contained implicitly in \cite[Proposition 7.6]{Koenig-Yang}, however, a proof is not explicitly given, so we give one for convenience.

\begin{lemma}\label{lem:approxmutationtriangle}
Let $(\sX, \sY)$ be an $\sS$-mutation pair.
\begin{enumerate}
\item For $x \in \sX$ consider the right mutation triangle $\Sigma^{-1} \Rmu{x} \rightlabel{f} s_x \rightlabel{g} \Sigma x \too \Rmu{x}$. The morphism $f \colon \Sigma^{-1} \Rmu{x} \to s_x$ is a minimal left $\extn{\sS}$-approximation.
\item For $y \in \sY$ consider the left mutation triangle $\Lmu{y} \too \Sigma^{-1} y \rightlabel{f} s^y \rightlabel{g} \Sigma \Lmu{y}$. The morphism $g \colon s^y \to \Sigma \Lmu{y}$ is a minimal right $\extn{\sS}$-approximation.
\end{enumerate}
\end{lemma}

\begin{proof}
We prove statement $(1)$, statement $(2)$ is analogous. We first claim that $\Sigma x$ contains no summands in $\sS$. Suppose, for a contradiction, that $\Sigma x = x' \oplus s$ for some $s \in \sS$, and observe that $\Hom(\Sigma x, s) = \Hom(x' \oplus s, s) \neq 0$ contradicts $\sX \subseteq {}\orth (\Sigma^{-1} \sS)$.

The fact that $x \in \sX$ and $\Hom (\sX,\sS) = 0$ immediately gives $f$ is a left $\extn{\sS}$-approximation, we just need to show minimality. Suppose $f$ is not left minimal. Then by Lemma~\ref{lem:basic-props}(2) we have the following isomorphism of distinguished triangles in $\sD$:
\[
\xymatrix@!R=8px{
\Sigma^{-1} \Rmu{x} \ar[r]^-{f} \ar[d]^-{\sim} & s_x \ar[r]^-{g} \ar[d]^-{\sim}                    & \Sigma x \ar[r] \ar[d]^-{\sim} & \Rmu{x} \ar[d]^-{\sim} \\
\Sigma^{-1} \Rmu{x} \ar[r]_-{\begin{pmat} f' \\0 \end{pmat}}   & s_1 \oplus s_2 \ar[r]                               & x' \oplus s_2 \ar[r]                 & \Rmu{x}.
}
\]
Thus, $s_2$ is a direct summand of $\Sigma x$. Hence, by above, $s_2 = 0$ and $f$ is a left minimal $\extn{\sS}$-approximation of $\Sigma^{-1} \Rmu{x}$, as claimed.
\end{proof}

The following straightforward lemma will be useful in shortening arguments throughout the article.

\begin{lemma} \label{lem:zeromaptoZ}
Suppose $(\sX,\sY)$ is an $\sS$-mutation pair.
Let $\trilabels{s}{d}{y}{\alpha}{\beta}{}$ be a triangle with $s \in \extn{\sS}$ and $y \in \sY$.
\begin{enumerate}
\item If $f,g \colon y \to y'$ are morphisms in $\sY$ such that $(f-g)\beta = 0$, then $f = g$.
\item If $\sigma, \tau \colon s' \to s$ are morphisms in $\extn{\sS}$ such that $\alpha(\sigma - \tau) = 0$, then $\sigma = \tau$.
\end{enumerate}
\end{lemma}

\begin{proof}
We prove the first statement, the second is analogous. By assumption, we have the following factorisation.
\[
\xymatrix{
s \ar[r] & d \ar[r]^-{\beta} \ar@{..>}[dr]_-{0} & y \ar[d]^-{f-g} \ar[r] & \Sigma s \ar@/^1pc/@{-->}[dl]^-{\exists} \\
           &                                                      & y'                            &
}
\]
Now $y' \in \sY \subset (\Sigma \sS)\orth$ implies that $f-g = 0$, as required.
\end{proof}

The next lemma will be used to define the shift functor in a pretriangulated category obtained from an $\sS$-mutation pair in the next section. 
Before stating it, we impose the blanket setup that will be used for the remainder of the article.

\begin{setup} \label{blanket-setup}
Let $\sS$ be an orthogonal collection of objects of $\sD$ such that $\extn{\sS}$ is functorially finite in $\sD$. Suppose one of the two conditions holds
\begin{enumerate}
\item $\sS$ is an $\SSS_{-1}$-subcategory; or,
\item $\Hom(\Sigma \sS,\sS) = 0$.
\end{enumerate}
\end{setup}

\begin{lemma} \label{lem:mutation}
Assume the hypotheses of Setup~\ref{blanket-setup}. Let $(\sX,\sY)$ be an $\sS$-mutation pair. The following assertions hold.
\begin{enumerate}
\item $\sX = \Lmu{\sY}$ and $\sY = \Rmu{\sX}$.
\item There is an equivalence of categories $G \colon \sX \to \sY$.
\end{enumerate}
\end{lemma}

\begin{proof}
For the first statement, we have 
\[
\Lmu{\sY} = \{ x \mid \Sigma^{-1} s^y \too x \rightlabel{\beta^y} \Sigma^{-1} y \rightlabel{\alpha^y} s^y \text{ for } y \in \sY \}, 
\]
where $\alpha^y$ is a minimal left $\extn{\sS}$-approximation.
By Lemma~\ref{lem:basic-props}(5), we have $\Lmu{\sY} \subseteq {}\orth \sS$. By definition, we have $\Lmu{\sY} \subseteq \Sigma^{-1} (\extn{\sS}*\sY)$. 

We need to show that $\Lmu{\sY} \subseteq {}\orth (\Sigma^{-1} \sS) \cap \sS \orth$. 
By applying $\Hom(-, \sS)$ to the triangle defining $x$ in $\Lmu{\sY}$, we have the following exact sequence:
\[
\xymatrix{
\Hom(y,\sS) \ar[r] &  \Hom(\Sigma x, \sS) \ar[r]^-{(\beta^y,\sS)} & \Hom(s^y, \sS) \ar[r]^-{(\alpha^y,\sS)} & \Hom(\Sigma^{-1} y, \sS) \ar[r] & 0.}
\]
By the dual of Lemma~\ref{lem:torsion-pair}(1), $\Hom(\alpha^y,\sS)$ is an isomorphism. Hence $\Hom(\beta^y,\sS) = 0$. But $y \in {}\orth \sS$, which implies that $\Hom(\beta^y,\sS)$ is a monomorphism, and so $x \in {}\orth (\Sigma^{-1} \sS)$.
Now, to show that $x \in \sS \orth$, first observe that $\Sigma^{-1} y \in \sS \orth$. If $\Hom(\Sigma \sS,\sS) = 0$ holds, then $x \in \sS\orth$ is immediate. If $\sS$ is an $\SSS_{-1}$-subcategory then applying the dual of Lemma~\ref{lem:torsion-pair}(3) also gives $x \in \sS\orth$.

For the reverse inclusion, suppose $x \in \sX$. Since $x \in \Sigma^{-1}(\extn{\sS} * \sY)$ there is a triangle
\begin{equation} \label{y-approx}
x \too \Sigma^{-1} y \rightlabel{f} s \too \Sigma x
\end{equation}
with $s \in \extn{\sS}$ and $y \in \sY$. The fact that $x \in {}\orth \sS$ means that $f \colon \Sigma^{-1} y \to s$ is a left $\extn{\sS}$-approximation. We claim that it is a minimal approximation. Indeed, if not, Lemma~\ref{lem:basic-props}(2) implies that \eqref{y-approx} is isomorphic to
\[
x \too \Sigma^{-1} y \rightlabel{\begin{pmat} f' \\ 0 \end{pmat}} s_1 \oplus s_2 \too \Sigma x,
\]
in which case $x \simeq x' \oplus \Sigma^{-1} s_2$ with $s_2 \in \extn{\sS}$ (see for example \cite[Lemma 3.1]{CSP}), giving a contradiction. A similar argument shows the statement for $\sY$. 

Now we turn to the second statement. We define a functor $G \colon \sX \to \sY$ as follows. 
For each $x \in \sX$ we fix a triangle
\[
s_x \rightlabel{\alpha_x} \Sigma x \rightlabel{\beta_x} Gx \rightlabel{\gamma_x} \Sigma s_x
\]
in which $\alpha_x$ is a minimal right $\extn{\sS}$-approximation. Note that $Gx \in \sY$ by part (1) of the lemma.

We now define $G$ on morphisms. Let $f \colon x \to x'$ in $\sX$. We explain below how to obtain the following commutative diagram from the morphism $\Sigma f$.
\[
\xymatrix@!R=8px{
\Sigma^{-1} Gx \ar@{-->}[d]_-{\Sigma^{-1} g} \ar[r]^-{-\Sigma^{-1} \gamma_x} & s_x \ar[r]^-{\alpha_x} \ar@{..>}[d]^-{\sigma} & \Sigma x \ar[d]^-{\Sigma f} \ar[r]^-{\beta_x} & Gx \ar@{-->}[d]^-{g} \\
\Sigma^{-1} Gx' \ar[r]_-{-\Sigma^{-1} \gamma_{x^\prime}}                                              & s_{x'} \ar[r]_-{\alpha_{x'}}                              & \Sigma x' \ar[r]_-{\beta_{x^\prime}}                           & Gx'
}
\]
First observe that the dotted arrow $\sigma \colon s_x \to s_{x'}$ exists because $\alpha_{x'}$ is a right $\extn{\sS}$-approximation; uniqueness follows by Lemma~\ref{lem:zeromaptoZ}(2). The map $g \colon Gx \to Gx'$ exists by $\TRthree$; uniqueness of $g$ follows from Lemma~\ref{lem:zeromaptoZ}(1). We therefore define $Gf = g$.
The functor $H \colon \sY \to \sX$ is defined dually. 

We now show that $HG \simeq 1_{\sX}$.
Let $x \in \sX$. Since $\alpha_x \colon s_x \to \Sigma x$ is a minimal right $\extn{\sS}$-approximation, it follows from Lemma~\ref{lem:approxmutationtriangle}(1) that $-\Sigma^{-1} \gamma_x \colon \Sigma^{-1} Gx \to s_x$ is a minimal left $\extn{\sS}$-approximation. 
Hence, we have a diagram in which the bottom row is the fixed minimal left $\extn{\sS}$-approximation triangle for $\Sigma^{-1} Gx$ used to construct $H$.
\[
\xymatrix@!R=8px{
x \ar@{-->}[d]_-{\varphi_x} \ar[r]^-{-\Sigma^{-1} \beta_x} & \Sigma^{-1} Gx \ar@{=}[d] \ar[r]^-{-\Sigma^{-1} \gamma_x} & s_x \ar@{-->}[d]_-{\theta_x} \ar[r]^-{\alpha_x} & \Sigma x \ar@{-->}[d]^-{\Sigma \varphi_x}\\
HGx \ar[r]_-{\beta^{Gx}} & \Sigma^{-1} Gx \ar[r]_-{\alpha^{Gx}} & s^{Gx} \ar[r]_-{\gamma^{Gx}} & \Sigma HGx
}
\]
By Lemma~\ref{lem:basic-props}(1), $\theta_x$ is an isomorphism; it is unique making the central square commute by the dual of Lemma~\ref{lem:zeromaptoZ}(2). It then follows that $\varphi_x$ exists and is an isomorphism; $\varphi_x$ is also unique making the left square commute by the dual of Lemma~\ref{lem:zeromaptoZ}(1). Hence $x \simeq HG x$.

Now, let $f \colon x \to x^\prime$ be a morphism in $\sX$. We need to show that $HG(f) \varphi_x = \varphi_{x^\prime} f$. The map $H(Gf)$ is defined by the following diagram:
\[
\xymatrix@!R=8px{
HG(x) \ar[d]_-{H(Gf)} \ar[r]^-{\beta^{Gx}} & \Sigma^{-1} Gx \ar[d]_-{\Sigma^{-1} Gf} \ar[r]^-{\alpha^{Gx}} & s^{Gx} \ar[d] \ar[r]^-{\gamma^{Gx}} & \Sigma HG(x) \ar[d] \\
HG(x^\prime) \ar[r]_-{\beta^{Gx^\prime}} & \Sigma^{-1} Gx^\prime \ar[r]_-{\alpha^{Gx^\prime}} & s^{Gx^\prime} \ar[r]_-{\gamma^{Gx^\prime}} & \Sigma HG(x^\prime).
}
\]
Now, $\beta^{Gx^\prime} HG(f) \varphi_x = (\Sigma^{-1} Gf) \beta^{Gx} \varphi_x = -(\Sigma^{-1} Gf) (\Sigma^{-1} \beta_x) = - (\Sigma^{-1} \beta_{x^\prime}) f = \beta^{Gx^\prime} \varphi_{x^\prime} f$. Hence, by the dual of Lemma~\ref{lem:zeromaptoZ}(1) it follows that $HG(f) \varphi_x = \varphi_{x^\prime} f$, as required. Similarly, we can show that $GH \simeq 1_\sY$. 
\end{proof}

\section{Pretriangulated categories from simple-minded mutation pairs} \label{sec:pretriangulated}

Throughout this section $\sD$ will be a Hom-finite $\kk$-linear triangulated category. The aim of this section is to establish the following theorem. Throughout the section we shall assume without further comment that the hypotheses of the theorem hold.

\begin{theorem} \label{thm:pretriangulated}
Assume the hypotheses of Setup~\ref{blanket-setup}. Let $\sZ$ be a subcategory of $\sD$ such that $(\sZ,\sZ)$ is an $\sS$-mutation pair satisfying, 
\begin{itemize}
\item[$\Zone$] $\sZ$ is closed under extensions and direct summands;  
\item[$\Ztwo$] the cones in $\sD$ of maps in $\sZ$ lie in $\extn{\sS} * \sZ$; and
\item[${\bf (Z2')}$] the cocones in $\sD$ of maps in $\sZ$ lie in $\sZ * \extn{\sS}$.
\end{itemize}
Then there is a functor $\extn{1} \colon \sZ \to \sZ$ and for each morphism $f \colon x \to y$ in $\sZ$ there is a diagram $x \rightlabel{f} y \too z_f \too x\extn{1}$ giving rise to a class of triangles $\Delta$ which makes $\sZ$ into a pretriangulated category.
\end{theorem}

Before proving the theorem, we define the functor $\extn{1}$ and the standard triangle $x \rightlabel{f} y \too z_f \too x\extn{1}$. We point out that the definition of $\extn{1}$ is the only place in the proof of Theorem~\ref{thm:pretriangulated} that we use the full force of the hypotheses of Setup~\ref{blanket-setup}.
 
\begin{definition}[Shift in $\sZ$] \label{def:shift-in-Z}
We define the shift functor $\extn{1} \coloneqq G \colon \sZ \to \sZ$, where $G$ is defined as in Lemma~\ref{lem:mutation}. The inverse shift functor $\extn{-1} = H$, as defined in Lemma~\ref{lem:mutation}.
\end{definition}

Before we can define the cones of morphisms in $\sZ$, we need the following observation.

\begin{lemma} \label{lem:cone-in-Z}
Let $f \colon x \to y$ be a morphism in $\sZ$ and consider the triangle 
$\trilabels{x}{y}{c_f}{f}{g_1}{h_1}$
in $\sD$. Let $\trilabels{s_f}{c_f}{z_f}{\alpha_f}{\beta_f}{\gamma_f}$ be the corresponding minimal right $\extn{\sS}$-approximation triangle. Then $z_f \in \sZ$.
\end{lemma}

\begin{proof}
By $\Ztwo$ there is a triangle $\trilabel{s}{c_f}{z}{\alpha}$ in $\sD$ with $s \in \extn{S}$ and $z \in \sZ$. Since $z \in \extn{\sS}\orth$, we have that $\alpha$ is a right $\extn{S}$-approximation of $c_f$. Hence, by Lemma~\ref{lem:basic-props}, $s_f$ is a summand of $s$ and $z_f$ is a summand of $z$ and lies in $\sZ$ by $\Zone$.
\end{proof}

Let $f \colon x \to y$ be a morphism in $\sZ$ and consider the triangle 
$\trilabels{x}{y}{c_f}{f}{g_1}{h_1}$ in $\sD$ together with the minimal right $\extn{\sS}$-approximation triangles of $c_f$ and $\Sigma x$ in the diagram below. 
\begin{equation} \label{eq:cone}
\xymatrix{& & s_f \ar[d]^{\alpha_f} \ar[r]^{\sigma} & s_x \ar[d]^{\alpha_x} \\
x \ar[r]^f & y \ar[r]^{g_1} \ar[dr]_g & c_f \ar[r]^{h_1} \ar[d]^{\beta_f} & \Sigma x \ar[d]^{\beta_x} \\
& & z_f \ar[r]^h \ar[d]^{\gamma_f} & x \extn{1} \ar[d]^{\gamma_x} \\
& & \Sigma s_f \ar[r]_{\Sigma \sigma} & \Sigma s_x,}
\end{equation}
The morphism $\sigma$ exists since $\alpha_x$ is a right $\extn{S}$-approximation of $\Sigma x$; it is unique by the fact that $x\extn{1} \in \sZ$ and $(\sZ,\sZ)$ is an $\sS$-mutation pair. 
The existence of $h$ follows from $\TRthree$ in $\sD$; it is unique making the squares commute by the same argument as above and Lemma~\ref{lem:cone-in-Z}. The object $z_f$ will be called the \emph{cone of $f$ in $\sZ$}. There is a natural dual construction of the \emph{cocone of $f$ in $\sZ$}.

\begin{definition}[Triangles in $\sZ$] \label{def:cones-in-Z}
Let $f \colon x \to y$ be a morphism in $\sZ$. The diagrams $\trilabelsZ{x}{y}{z_f}{f}{g}{h}$ will be called the \emph{standard triangles} of $\sZ$. We define $\Delta$, the set of diagrams of the form $\triZ{x}{y}{z}$ with $x, y, z \in \sZ$ isomorphic to a standard triangle, to be a set of triangles in $\sZ$.
\end{definition}

Before proving Theorem~\ref{thm:pretriangulated} we record an observation that will be useful later.

\begin{lemma} \label{lem:cone-orth}
Let $f \colon x \to y$ be a morphism in $\sZ$. Then $c_f \in (\Sigma \sS)\orth$.
\end{lemma}

\begin{proof}
Simply apply the functor $\Hom_\sD(\Sigma \sS,-)$ to the triangle $\trilabel{x}{y}{c_f}{f}$ in $\sD$ and use the fact that $x,y \in \sZ$.
\end{proof}

\begin{proof}[Proof of Theorem~\ref{thm:pretriangulated}]
To show that $\sZ$ is a pretriangulated category with the given pretriangulated structure we must verify axioms $\TRone$, $\TRtwo$ and $\TRthree$. The verification of $\TRone$ is immediate.

\medskip
\noindent
For $\TRtwo$ it is sufficient to show that for a standard triangle $\trilabelsZ{x}{y}{z_f}{f}{g}{h}$ the diagram $\trilabelsZ{y}{z_f}{x\extn{1}}{g}{h}{-f\extn{1}}$ is isomorphic to a standard triangle $\trilabelsZ{y}{z_f}{z_g}{g}{a}{b}$. This will establish $\TRtwo$ in one direction; the other direction is analogous. 

Recall the diagram \eqref{eq:cone} defining the cone of $f$ in $\sZ$ and consider the octahedral diagram coming from the composition $g= \beta_f g_1$.
\begin{equation}\label{eq:TR2-1}
\xymatrix{
                         & s_f \ar[d]_{\alpha_f} \ar@{=}[r]   & s_f \ar[d]^{h_1 \alpha_f}                 & \\
y \ar@{=}[d] \ar[r]^{g_1} & c_f \ar[r]^{h_1} \ar[d]_{\beta_f}   & \Sigma x \ar[r]^{-\Sigma f} \ar[d]^{\beta} & \Sigma y \ar@{=}[d] \\
y \ar[r]_g               & z_f \ar[r]_{a_1} \ar[d]_{\gamma_f}  & c_g \ar[r]_{b_1} \ar[d]^{\gamma}          & \Sigma y \\
                         & \Sigma s_f \ar@{=}[r]           & \Sigma s_f                             &
}
\end{equation}
We now define the cone of $g$ in $\sZ$.
\begin{equation}\label{eq:TR2-2}
\xymatrix{
            &                           & s_g \ar[r]^{\tau} \ar[d]_{\alpha_g}  & s_y \ar[d]^{\alpha_y} \\
y \ar[r]^g  & z_f \ar[r]^{a_1} \ar[dr]_a  & c_g \ar[r]^{b_1} \ar[d]^{\beta_g}    & \Sigma y \ar[d]^{\beta_y} \\
            &                           & z_g \ar[r]^b \ar[d]^{\gamma_g}      & y \extn{1} \ar[d]^{\gamma_y} \\
            &                           & \Sigma s_g \ar[r]_{\Sigma \tau}     & \Sigma s_y
}
\end{equation}
Applying $\TRthree$ in $\sD$ to the second vertical triangle in diagram \eqref{eq:TR2-1} and the defining triangle of $x \extn{1}$, we get the following commutative diagram.
\begin{equation}\label{eq:TR2-3}
\xymatrix@!R=8px{
s_\alpha \ar[r]^{h_1 \alpha_f} \ar[d]_{\sigma} & \Sigma x \ar[r]^{\beta} \ar@{=}[d] & c_g \ar[r]^{\gamma} \ar[d]_{k^\prime} & \Sigma s_f \ar[d]_{\Sigma \sigma} \\
s_x \ar[r]_{\alpha_x}                      & \Sigma x \ar[r]_{\beta_x}           & x \extn{1} \ar[r]_{\gamma_x}       & \Sigma s_x
}
\end{equation}
Recall that $\sigma$ is the unique map such that $\alpha_x \sigma = h_1 \alpha_f$. By the same argument, the map $k' \colon c_g \to x \extn{1}$ is the unique completion to a morphism of triangles given by $\TRthree$ in $\sD$.

Note that $k^\prime \alpha_g = 0$ since $\sZ \subset \sS\orth$. Therefore, there is a map $k \colon z_g \to x \extn{1}$ such that $k^\prime = k \beta_g$. We claim that the following diagram commutes.
\begin{equation}\label{eq:iso}
\xymatrix@!R=8px{
y \ar[r]^g \ar@{=}[d] & z_f \ar[r]^a \ar@{=}[d] & z_g \ar[r]^b \ar[d]_{k}       & y \extn{1} \ar@{=}[d] \\
y \ar[r]_g            & z_f \ar[r]_h            & x \extn{1} \ar[r]_{-f \extn{1}} & y \extn{1}
}
\end{equation}
It is clear that the left-hand square commutes. For the central square, we have:
\[
h \beta_f \stackrel{\eqref{eq:cone}}{=} \beta_x h_1
          \stackrel{\eqref{eq:TR2-3}}{=} k' \beta h_1
          \stackrel{\eqref{eq:TR2-1}}{=} k' a_1 \beta_f
          =                              k \beta_g a_1 \beta_f
          \stackrel{\eqref{eq:TR2-2}}{=} k a \beta_f,
\]
so that $(h - ka) \beta_f = 0$. Therefore, by Lemma~\ref{lem:zeromaptoZ}, $h = ka$, showing that the central square of diagram \eqref{eq:iso} commutes.

For the right-hand square, we have,
\[
b \beta_g \beta \stackrel{\eqref{eq:TR2-2}}{=} \beta_y b_1 \beta
                \stackrel{\eqref{eq:TR2-1}}{=} \beta_y (-\Sigma f) 
                =                              (-f \extn{1}) \beta_x
                \stackrel{\eqref{eq:TR2-3}}{=} (-f \extn{1}) k' \beta
                =                              (-f \extn{1}) k \beta_g \beta,
\]
where the middle equality is by the definition of $f \extn{1}$.
Hence, $(b + (f \extn{1})k) \beta_g \beta = 0$. 
Now two applications of Lemma~\ref{lem:zeromaptoZ} shows that $b = -f \extn{1} k$, giving the commutativity of the right-hand square of diagram \eqref{eq:iso}.

Finally, if $k$ is an isomorphism, we have the required isomorphism of diagrams. To show this, consider the composition $\beta_g \beta \colon \Sigma x \to z_g$. Applying the Octahedral Axiom in $\sD$ to this composition shows that the cone $c_{\beta_g\beta} \simeq \Sigma s$ for some $s \in \extn{\sS}$, giving rise to the triangle,
\[
\trilabels{s}{\Sigma x}{z_g}{\tilde{\alpha}}{\beta_g\beta}{\tilde{\gamma}},
\]
in $\sD$. We claim that $\tilde{\alpha} \colon s \to \Sigma x$ is a minimal right $\extn{\sS}$-approximation. Since $z_g \in \sZ \subset \sS\orth$, $\tilde{\alpha}$ is clearly a right $\extn{\sS}$-approximation. Suppose that $\tilde{\alpha}$ is not right minimal. Then $s \simeq s_x \oplus s^\prime$, for some $s^\prime \in \extn{\sS}$, and $z_g \simeq x\extn{1} \oplus \Sigma s^\prime$. Therefore, $\Hom_\sD(\Sigma s^\prime, z_g) \neq 0$, contradicting the fact that $(\sZ,\sZ)$ is an $\sS$-mutation pair. 
Hence, we have a commutative diagram, as follows,
\[
\xymatrix@!R=10pt{
s \ar[r]^{\tilde{\alpha}} \ar[d]_{\simeq} & \Sigma x \ar[r]^{\beta_g \beta} \ar@{=}[d] & z_g \ar[r]^{\tilde{\gamma}} \ar[d]_{\simeq}^{\tilde{k}} & \Sigma s \ar[d]_{\simeq} \\
s_x \ar[r]_{\alpha_x}                  & \Sigma x \ar[r]_{\beta_x}                 & x \extn{1} \ar[r]_{\gamma_x}                      & \Sigma s_x},
\]
where by the usual argument $\tilde{k}$ is unique making the middle square commute. Now $\beta_x = k^\prime \beta = k \beta_g \beta$. Hence $k = \tilde{k}$ is an isomorphism, and so 
$\trilabelsZ{y}{z_f}{x \extn{1}}{g}{h}{-f \extn{1}} \in \Delta$. 

\medskip
\noindent
We now turn to the verification of $\TRthree$ in $\sZ$. It is enough to show for two standard triangles, indicated below, in which the left-hand square commutes, then there exists a third arrow $c$ making the whole diagram commute.
\begin{equation} \label{eq:TR3}
\xymatrix@!R=10pt{
x \ar[r]^f \ar[d]^a       & y \ar[r]^g \ar[d]^b      & z_f \ar[r]^h \ar@{-->}[d]^c  & x \extn{1} \ar[d]^{f \extn{1}} \\
x^\prime \ar[r]_-{f^\prime} & y^\prime \ar[r]_-{g^\prime} & z_{f^\prime} \ar[r]_-{h^\prime}  &  x^\prime \extn{1}
}
\end{equation}
We require the following diagrams; the notation is set up as in Definitions \ref{def:shift-in-Z} and \ref{def:cones-in-Z}. 
\[
(A) \quad
\xymatrix@!R=10pt{
x \ar[r]^-{f} \ar[d]^-{a} & y \ar[r]^-{g_1} \ar[d]^-{b} & c_f \ar[r]^-{h_1} \ar@{-->}[d]^-{c_1} & \Sigma x \ar[d]^-{\Sigma a} \\
x^\prime \ar[r]_-{f^\prime}       & y^\prime \ar[r]_-{g^\prime_1} & c_{f^\prime_1} \ar[r]_-{h^\prime_1}       & \Sigma x^\prime
}
\qquad (B) \quad
\xymatrix@!R=10pt{
s_f \ar[r]^-{\alpha_f} \ar[d]    & c_f \ar[r]^-{\beta_f} \ar[d]^-{c_1} & z_f \ar[r]^-{\gamma} \ar@{-->}[d]^-{c} & \Sigma s_f \ar[d] \\
s_{f^\prime} \ar[r]_-{\alpha_{f'}}  & c_{f^\prime} \ar[r]_-{\beta_{f^\prime}}  & z_{f^\prime} \ar[r]_-{\gamma_{f'}}         & \Sigma s_{f^\prime}.
}
\]
Diagram $(A)$ is simply $\TRthree$ in $\sD$ applied to the cone of $f$ and $f'$ triangles in $\sD$. Diagram $(B)$ takes the morphism $c_1$ obtained in diagram $(A)$ and is constructed in an analogous manner to the vertical part of diagram \eqref{eq:cone}.
Finally, the following diagram follows by the definition of $a \extn{1}$ in Definition \ref{def:shift-in-Z}.
\begin{equation}\label{eq:C}
\xymatrix@!R=10pt{
s_x \ar[r]^-{\alpha_x} \ar[d]_-{\sigma} & \Sigma x \ar[r]^-{\beta_x} \ar[d]^-{\Sigma a} & x \extn{1} \ar[r]^-{\gamma_x} \ar[d]^-{a \extn{1}} & \Sigma s_x \ar[d] \\
s_{x^\prime} \ar[r]_-{\alpha_{x^\prime}}      & \Sigma x^\prime \ar[r]_-{\beta_{x'}} & x^\prime \extn{1} \ar[r]_-{\gamma_{x'}}                         & \Sigma s_{x^\prime}.
}
\tag{$C$}
\end{equation}

We first show that the central square of diagram \eqref{eq:TR3} commutes. We have, $g = \beta_f g_1$ and $h \beta_f = \beta_x h_1$ from diagram \eqref{eq:cone}, and $g' = \beta_{f'} g'_1$ and $h' \beta_{f'} = \beta_{x'} h'_1$ from the corresponding diagram from $f' \colon x' \to y'$. It follows that
\[
c g =                           c \beta_f g_1
    \stackrel{(B)}{=} \beta_{f'} c_1 g_1
    \stackrel{(A)}{=} \beta_{f'} g'_1 b
    =                 g' b. 
\]
For the right-hand square of diagram \eqref{eq:TR3}, we have
\[
h' c \beta_f \stackrel{(B)}{=}           h' \beta_{f'} c_1
             =                           \beta_{x'} h'_1 c_1
             \stackrel{(A)}{=}           \beta_{x'} (\Sigma a) h_1
             \stackrel{\eqref{eq:C}}{=}  (a \extn{1}) \beta_x h_1
             =                           (a \extn{1}) h \beta_f,
\]
from which it follows that $(h'c - a\extn{1}h)\beta_f = 0$. 
Applying Lemma~\ref{lem:zeromaptoZ}, we obtain that $h' c = a\extn{1} h$ and diagram \eqref{eq:TR3} commutes.

We therefore conclude that $\sZ$ with the pretriangulated structure given by $\extn{1} \colon \sZ \to \sZ$ and $\Delta$ is a pretriangulated category, completing the proof of Theorem~\ref{thm:pretriangulated}. 
\end{proof}

\section{The Octahedral Axiom} \label{sec:triangulated}

In this section we show that the Octahedral Axiom also holds for the pretriangulated structure defined in Theorem~\ref{thm:pretriangulated}.

\begin{theorem} \label{thm:triangulated}
Assume the hypotheses of Setup~\ref{blanket-setup} and Theorem~\ref{thm:pretriangulated} hold.
Then the pretriangulated structure of Theorem~\ref{thm:pretriangulated} on $\sZ$ satisfies the Octahedral Axiom.
\end{theorem}

\begin{proof}
We need to verify that the Octahedral Axiom holds for the triangulated structure defined on $\sZ$ in Section~\ref{sec:pretriangulated}.
Let $\trilabelsZ{u}{v}{z_f}{f}{g}{h}$ and $\trilabelsZ{v}{w}{z_a}{a}{b}{c}$ be two standard triangles in $\sZ$.
We claim there is a commutative diagram of the following form in which each row and column is isomorphic to a standard triangle in $\sZ$ and $f\extn{1}q = cs$.
\begin{equation} \label{eq:octahedral}
\xymatrix{
  & z_a \extn{-1}  \ar@
  {=}[r] \ar[d]_{-c \extn{-1}} & z_a \extn{-1} \ar[d]^{-t \extn{-1}}   &                      \\
u \ar[r]^f \ar@{=}[d] & v \ar[r]^{g} \ar[d]_a                         & z_f \ar[r]^{h} \ar[d]^{r}           &  u\extn{1} \ar@{=}[d] \\
u \ar[r]_{af}          & w \ar[r]_{p} \ar[d]_{b}                       & z_{af} \ar[r]_{q} \ar[d]^{s}         & u\extn{1} \\
                      & z_a \ar@{=}[r]                               & z_a                                & 
}
\end{equation}
We first observe that the two rows are standard triangles in $\sZ$ by construction and the left-hand column is a standard triangle by $\TRtwo$ in $\sZ$. 
We break the rest of the proof up into three steps. Firstly, we define the maps $r$ and $s$ occurring in diagram \eqref{eq:octahedral}. Secondly, we show that diagram \eqref{eq:octahedral} commutes and $f\extn{1}q = cs$. Finally, in the most involved step, we show that the sequence $\trilabelsZ{z_f}{z_{af}}{z_a}{r}{s}{t}$ is isomorphic to a standard triangle in $\sZ$.

\Step{1} {\it The construction of diagram \eqref{eq:octahedral}, in particular, the maps $r$ and $s$.}

\medskip
\noindent
Let $\trilabels{u}{v}{c_f}{f}{g_1}{h_1}$ and $\trilabels{v}{w}{c_a}{a}{b_1}{c_1}$ be triangles in $\sD$ defined by $f$ and $a$. By the Octahedral Axiom in $\sD$, we have the following commutative diagram,
\begin{equation} \label{eq:octahedron-in-D}
\xymatrix{
                      & \Sigma^{-1} c_a \ar@{=}[r] \ar[d]_{-\Sigma^{-1} c_1} & \Sigma^{-1} c_a \ar[d]^{-\Sigma^{-1} t_1}  & \\
u \ar[r]^f \ar@{=}[d] & v \ar[r]^{g_1} \ar[d]_a                          & c_f \ar[r]^{h_1} \ar[d]^{r_1}           & \Sigma u \ar@{=}[d] \\
u \ar[r]_{af}         & w \ar[r]_{p_1} \ar[d]_{b_1}                        & c_{af} \ar[r]_{q_1} \ar[d]^{s_1}        & \Sigma u \\
                     & c_a \ar@{=}[r]                                   & c_a                                  &
}
\end{equation}
such that $(\Sigma f)q_1 = c_1 s_1$.
Considering the approximation triangles defining $z_f$, $z_{af}$ and $z_a$, we obtain the following commutative diagram.
\begin{equation} \label{eq:define-r-s}
\xymatrix@!R=8px{
s_f \ar[r]^-{\alpha_f} \ar@{-->}[d]_-{\sigma}   & c_f \ar[r]^-{\beta_f} \ar[d]_-{r_1}     & z_f \ar[r]^-{\gamma_f} \ar@{-->}[d]^-{r}      & \Sigma s_f \ar@{-->}[d]^-{\Sigma \sigma} \\
s_{af} \ar[r]_-{\alpha_{af}} \ar@{-->}[d]_-{\tau} & c_{af} \ar[r]^-{\beta_{af}} \ar[d]_-{s_1} & z_{af} \ar[r]^-{\gamma_{af}} \ar@{-->}[d]^-{s} & \Sigma s_{af} \ar@{-->}[d] \\
s_a \ar[r]_-{\alpha_a}                          & c_a \ar[r]_-{\beta_a}                    & z_a \ar[r]_-{\gamma_a}                       & \Sigma s_a
}
\end{equation}
Since $\alpha_{af}$ and $\alpha_a$ are right $\extn{\sS}$-approximations, the vertical morphisms $\sigma$ and $\tau$ exist making the left-hand squares commute. Moreover, by the dual of Lemma~\ref{lem:zeromaptoZ}, they are unique making those squares commute. Therefore, by $\TRthree$ in $\sD$, the vertical morphisms $r$ and $s$ exist; applying Lemma~\ref{lem:zeromaptoZ} shows that they are unique making the central squares commute.
We have now constructed diagram \eqref{eq:octahedral}.

\Step{2} {\it Diagram \eqref{eq:octahedral} commutes and $f\extn{1} q = cs$.}

\medskip
\noindent
Clearly, the topmost and leftmost squares of diagram \eqref{eq:octahedral} commute. Recall from Definition~\ref{def:cones-in-Z} that we have,
\begin{equation} \label{eq:commutes}
g= \beta_f g_1, \quad
h \beta_f = \beta_{u} h_1, \quad
b= \beta_a b_1, \quad 
c \beta_a = \beta_{v} c_1, \quad
p = \beta_{af} p_1, \quad  
q \beta_{af} = \beta_{u} q_1.
\end{equation}
It is now clear that the middle square commutes:
\[
pa \stackrel{\eqref{eq:commutes}}{=}        \beta_{af} p_1 a
   \stackrel{\eqref{eq:octahedron-in-D}}{=} \beta_{af} r_1 g_1
   \stackrel{\eqref{eq:define-r-s}}{=}      r \beta_f g_1 
   \stackrel{\eqref{eq:commutes}}{=}        rg.
\]
Similarly, for the bottommost square we have:
\[
sp \stackrel{\eqref{eq:commutes}}{=}        s \beta_{af} p_1 
   \stackrel{\eqref{eq:define-r-s}}{=}      \beta_a s_1 p_1 
   \stackrel{\eqref{eq:octahedron-in-D}}{=} \beta_a b_1 
   \stackrel{\eqref{eq:commutes}}{=} b.
\]
Finally, for the rightmost square, we get a similar chain of equalities:
\[
q r \beta_f \stackrel{\eqref{eq:define-r-s}}{=}      q \beta_{af} r_1 
            \stackrel{\eqref{eq:commutes}}{=}        \beta_{u} q_1 r_1 
            \stackrel{\eqref{eq:octahedron-in-D}}{=} \beta_{u} h_1 
            \stackrel{\eqref{eq:commutes}}{=}        h \beta_f.
\]
Thus, $(qr-h) \beta_f = 0$, and the rightmost square commutes by Lemma \ref{lem:zeromaptoZ}. 
Finally, we have
\[
c s \beta_{af} \stackrel{\eqref{eq:define-r-s}}{=}       c \beta_a s_1
              \stackrel{\eqref{eq:commutes}}{=}         \beta_v c_1 s_1
              \stackrel{\eqref{eq:octahedron-in-D}}{=}  \beta_v (\Sigma f)q_1
              =                                         f\extn{1} \beta_u q_1
              \stackrel{\eqref{eq:commutes}}{=}         f\extn{1} q \beta_{af},
\]
where the unmarked equality follows by definition of $f\extn{1}$ in Definition~\ref{def:shift-in-Z}. Therefore $(cs-f\extn{1}q)\beta_{af} = 0$, so that by Lemma~\ref{lem:zeromaptoZ} we have $cs = f\extn{1} q$.

\medskip

We now show that the sequence $\trilabelsZ{z_f}{z_{af}}{z_a}{r}{s}{t}$ is isomorphic to a standard triangle. We start by constructing the standard triangle in $\sZ$ corresponding to the map $r \colon z_f \to z_{af}$. For this we will have to choose a specific triangle occurring in a $3\times 3$ diagram.

\Step{3} {\it There is a $3\times 3$ diagram in which each square is commutative except the bottom right-hand square, which is anticommutative:
\begin{equation} \label{3by3}
\xymatrix@!R=8px{
s_f \ar[r]^{\alpha_f} \ar[d]_-{\sigma} & c_f \ar[r]^{\beta_f} \ar[d]^{r_1}       & z_f \ar[r]^{\gamma_f} \ar[d]^r        & \Sigma s_f \ar[d] \\
s_{af} \ar[r]_{\alpha_{af}} \ar[d]       & c_{af} \ar[r]_{\beta_{af}} \ar[d]^{s_1}  & z_{af} \ar[r]_{\gamma_{af}} \ar[d]^{s'}  & \Sigma s_{af} \ar[d] \\
c_\sigma \ar[d] \ar[r]                & c_a \ar[r]_-{\kappa} \ar[d]^{t_1}    & c_r \ar[r] \ar[d]^{t'}               & \Sigma c_\sigma \ar[d] \\
\Sigma s_f \ar[r]_{\Sigma \alpha_f}      & \Sigma c_f \ar[r]_{\Sigma \beta_f}     & \Sigma z_f \ar[r]_{\Sigma \gamma_f}     & \Sigma^2 s_f.
 }
\end{equation}}

\medskip
\noindent
Taking the top left-hand square of diagram \eqref{eq:define-r-s}, we can apply \cite[Proposition 1.1.11]{BBD} to obtain the $3\times 3$ diagram \eqref{3by3}. Note that in the proof of \cite[Proposition 1.1.11]{BBD}, the triangles corresponding to the two top rows and two left-hand columns can be freely chosen. Therefore, given this choice, as noted in Step 1, the uniqueness of $r$ making the top two right-hand squares commute forces the morphim in this position in the $3\times 3$ diagram to be $r$. 

\medskip

We now consider the triangle comprising the third column of diagram \eqref{3by3} and construct the corresponding triangle in $\sZ$ according to Definition~\ref{def:cones-in-Z}.
\begin{equation} \label{eq:cone-r}
\xymatrix@!R=8px{
             &                                       & s_r \ar[r] \ar[d]^{\alpha_r}           & s_{\Sigma z_f} \ar[d]^{\alpha_{z_f}} \\
z_f \ar[r]^r & z_{af} \ar[r]^{s^\prime} \ar[dr]_{\tilde{s}} & c_r \ar[r]^{t^\prime} \ar[d]^{\beta_r}   & \Sigma z_f \ar[d]^{\beta_{z_f}}  \\
             &                                       & z_r \ar[r]^{\tilde{t}} \ar[d]^{\gamma_r} & z_f \extn{1} \ar[d]^{\gamma_{z_f}} \\
             &                                       & \Sigma s_r \ar[r]                    & \Sigma s_{\Sigma z_f}
}
\end{equation}

\Step{4} {\it There is a morphism $\zeta \colon z_a \to z_r$ such that the following diagram commutes.
\begin{equation}\label{eq:isooftriangles}
\xymatrix@!R=8px{
z_f \ar[r]^r \ar@{=}[d] & z_{af} \ar[r]^s \ar@{=}[d] & z_a \ar[r]^t \ar[d]^{\zeta} & z_f\extn{1} \ar@{=}[d] \\
z_f \ar[r]_r            & z_{af} \ar[r]_{\tilde{s}}    & z_r \ar[r]_{\tilde{t}}            & z_f\extn{1}}
\end{equation}}

\medskip
\noindent
Consider the morphism $\kappa \colon c_a \to c_r$ occurring in diagram \eqref{3by3}. Since $\sZ \in \sS\orth$, we have $(\beta_r \kappa) \alpha_a = 0$. Therefore we have the following factorisation.
\begin{equation} \label{eq:define-zeta}
\xymatrix@!R=8px{
s_a \ar[r]^-{\alpha_a} \ar@{..>}[dr]_-{0} & c_a \ar[r]^-{\beta_a} \ar[d]^-{\beta_r\kappa} & z_a \ar[r]^-{\gamma_a} \ar@/^1pc/@{-->}[dl]^-{\zeta} & \Sigma s_a \\
                                         & z_r                                          &                                              &
}
\end{equation}
We now need to check that $\zeta$ makes diagram \eqref{eq:isooftriangles} commute.
To see that the central square of \eqref{eq:isooftriangles} commutes, we have the following sequence of equalities,
\[
\zeta s \beta_{af} \stackrel{\eqref{eq:define-r-s}}{=}   \zeta \beta_a s_1 
                  \stackrel{\eqref{eq:define-zeta}}{=}  \beta_r \kappa s_1 
                  \stackrel{\eqref{3by3}}{=}            \beta_r s^\prime \beta_{af}.
\]
Hence $(\zeta s - \beta_r s')\beta_{af} = 0$, so that  $\zeta s= \beta_r s^\prime = \tilde{s}$ by Lemma~\ref{lem:zeromaptoZ}. 

For the commutativity of the right-hand square of diagram \eqref{eq:isooftriangles}, we have,
\begin{align*}
\tilde{t} \zeta \beta_a \stackrel{\eqref{eq:define-zeta}}{=} &   \tilde{t} \beta_r \kappa
                        \stackrel{\eqref{eq:cone-r}}{=}          \beta_{z_f} t^\prime \kappa
                        \stackrel{\eqref{3by3}}{=}               \beta_{z_f} (\Sigma \beta_f) t_1      
                        \stackrel{\eqref{eq:octahedron-in-D}}{=} \beta_{ z_f} (\Sigma \beta_f) (\Sigma g_1) c_1 \\
                  &     \stackrel{\eqref{eq:commutes}}{=}         \beta_{z_f} (\Sigma g) c_1 
                        =                                         g\extn{1} \beta_{v} c_1        
                        \stackrel{\eqref{eq:commutes}}{=}         g\extn{1} c \beta_a  
                        \stackrel{\eqref{eq:octahedral}}{=}       t \beta_a,
\end{align*}
where the unlabelled equality follows by the definition of $g\extn{1}$ in Definition~\ref{def:shift-in-Z}.
It therefore follows that $(\tilde{t} \zeta - t)\beta_a = 0$, so that Lemma~\ref{lem:zeromaptoZ} implies $\tilde{t} \zeta = t$. Therefore, diagram \eqref{eq:isooftriangles} commutes.

\Step{5} {\it The morphism $\zeta \colon z_a \to z_r$ in diagram \eqref{eq:isooftriangles} is an isomorphism.}

\medskip
\noindent

In order to show that $\zeta \colon z_a \to z_r$ is an isomorphism we shall need the following lemma, which asserts that the cone of the morphism $\sigma \colon s_f \to s_{af}$ in diagram \eqref{3by3} lies in $\extn{\sS}$. As its proof is quite involved, we defer the proof of the lemma until after we have completed the proof of Theorem~\ref{thm:triangulated}.

\begin{lemma} \label{lem:cone-of-sigma}
In the $3 \times 3$ diagram \eqref{3by3} above, the object $c_\sigma \in \extn{S}$.
\end{lemma}

Consider the octahedral diagram in $\sD$ coming from the composition $\beta_r \kappa$.
\[
\xymatrix@!R=8px{
                             & s_r \ar@{=}[r] \ar[d]       & s_r \ar[d]                           &\\
c_a \ar[r]^{\kappa} \ar@{=}[d] & c_r \ar[r] \ar[d]^{\beta_r}   & \Sigma c_\sigma \ar[r] \ar[d]         & \Sigma c_a \ar@{=}[d] \\
c_a \ar[r]_{\beta_r \kappa}      & z_r \ar[r] \ar[d]           & \Sigma s \ar[r]_{\Sigma \alpha} \ar[d]  & \Sigma c_a \\
                             & \Sigma s_r \ar@{=}[r]       & \Sigma s_r                           &
}
\] 
Since, by Lemma~\ref{lem:cone-of-sigma}, $c_\sigma \in \extn{\sS}$, we have $s \in \extn{\sS}$. 

We now claim that $\alpha \colon s \to c_a$ is a minimal right $\extn{\sS}$-approximation of $c_a$.
Since $z_r \in \sZ \subset \sS\orth$, it is clear that $\alpha$ is a right $\extn{\sS}$-approximation. By Lemma~\ref{lem:basic-props}, $s \simeq s_a \oplus s'$ for some $s' \in \extn{\sS}$, and $z_r \simeq z_a \oplus \Sigma s'$. But since $\sZ \subset (\Sigma \sS)\orth$ it follows that $s' = 0$, i.e. $s \simeq s_a$ and $\alpha$ is minimal. 

Now consider the following diagram.
\[
\xymatrix@!R=8px{
s_a \ar[r]^{\alpha_a} \ar@{-->}[d]_{\pi} & c_a \ar[r]^{\beta_a} \ar@{=}[d]       & z_a \ar[r]^{\gamma_a} \ar[d]^{\zeta} & \Sigma s_a \ar@{-->}[d]^-{\Sigma \pi} \\
s \ar[r]_{\alpha} \ar@{-->}[d]_{\pi'}    & c_a \ar[r]_{\beta_r \kappa} \ar@{=}[d]  & z_r \ar[r] \ar@{-->}[d]^{\zeta'}   & \Sigma s \ar@{-->}[d]^-{\Sigma \pi'} \\
s_a \ar[r]_{\alpha_a}                   & c_a \ar[r]_{\beta_a}                  & z_a \ar[r]_{\gamma_a}               & \Sigma s_a
}
\]
By Step 4, $\zeta \beta_a = \beta_r \kappa$, so that $\TRthree$ in $\sD$ asserts the existence of the morphism $\pi \colon s_a \to s$. Using the fact that $\alpha_a \colon s_a \to c_a$ is a right $\extn{\sS}$-approximation, we obtain the existence of a morphism $\pi' \colon s \to s_a$ making the diagram commute. Finally, applying $\TRthree$ again provides the morphism $\zeta' \colon z_r \to z_a$.

Now, right minimality of $\alpha_a$ implies that $\pi' \pi$ is an isomorphism, whence it follows that $\zeta' \zeta$ is an isomorphism and therefore $\zeta$ is a split monomorphism. 
Similarly, right minimality of $\alpha$ implies that $\pi \pi'$ is an isomorphism and therefore so too is $\zeta \zeta'$, in which case $\zeta$ is a split epimorphism. 
Hence $\zeta \colon z_a \to z_r$ is an isomorphism, as claimed. 
This completes the proof of Theorem~\ref{thm:triangulated}.
\end{proof}

\begin{remark}
Note that if we additionally assume $\sZ \subset (\Sigma^2 \sS)\orth$ one can obtain that $\zeta \colon z_a \to z_r$ is an isomorphism avoiding Lemma~\ref{lem:cone-of-sigma}. This assumption is benign when $\sD$ is $(-w)$-Calabi-Yau for $w \geq 2$ and $\sS$. However, it is in general false when $w = 1$.
\end{remark}

\begin{proof}[Proof of Lemma~\ref{lem:cone-of-sigma}]
We note that the proof of this lemma also requires the full force of the hypotheses of Setup~\ref{blanket-setup}.
First observe that if $s_f = 0$ then $c_\sigma \simeq s_{af} \in \extn{\sS}$, so we may assume that $s_f \neq 0$.
The strategy is to use Lemma~\ref{lem:Dugas}.
The argument is rather intricate so we proceed in a sequence of steps.

\Step{1} 
{\it If $s_f \neq 0$ then $\sigma \neq 0$.}

\medskip
\noindent
Suppose $r_1 \alpha_f = 0$ and consider the octahedral diagram arising from this composition.
\[
\xymatrix{
                              & \Sigma^{-1} c_a \ar[d] \ar@{=}[r]   & \Sigma^{-1} c_a \ar[d]         & \\
s_f \ar@{=}[d] \ar[r]^-{\alpha_f} & c_f \ar[d]_-{r_1} \ar[r]^-{\beta_f} & z_f \ar[d] \ar[r]^-{\gamma_f} & \Sigma s_f \ar@{=}[d] \\
s_f \ar[r]_-{0}               & c_{af} \ar[d] \ar[r]                & c_{af} \oplus \Sigma s_f \ar[d] \ar[r] & \Sigma s_f \\
                             & c_a \ar@{=}[r]                      & c_a                         & }
\]
Now $z_f \in \sZ \subset (\Sigma \sS)\orth$ and $c_a \in (\Sigma \sS)\orth$ by Lemma~\ref{lem:cone-orth}, which forces $\Sigma s_f \in (\Sigma \sS)\orth$. Hence we obtain $s_f = 0$, a contradiction. Thus $\alpha_{af} \sigma = r_1 \alpha_f \neq 0$. In particular $\sigma \neq 0$.

\Step{2}
{\it If $s_f \neq 0$ then $\Hom_\sD(\extn{\sS},r_1) \colon \Hom_\sD(\extn{\sS},c_f) \to \Hom_\sD(\extn{\sS},c_{af})$ is injective.}

\medskip
\noindent
By Step 1 we also know that $r_1 \neq 0$.
Now let $s \in \extn{\sS}$ and suppose $\pi \colon s \to c_f$ satisfies $r_1 \pi = 0$. We therefore have the following factorisation,
\[
\xymatrix{
                                & \ar@{-->}[dl]_-{\pi'} s \ar[d]^-{\pi} \ar@{..>}[dr]^-{0} &                    & \\
\Sigma^{-1} c_a \ar[r]_-{-\Sigma t_1} & c_f \ar[r]_-{r_1}                                        & c_{af} \ar[r]_-{s_1} & c_a .
}
\]
But $\pi' = 0$ since $c_a \in (\Sigma \sS)\orth$. Thus $\pi = 0$ and $\Hom_\sD(\extn{\sS},r_1)$ is injective.

Note that by Step 1, if $s_f$ has $\sS$-length one, then $c_\sigma \in \extn{\sS}$ by Lemma~\ref{lem:Dugas}. Therefore, we now assume that $s_f$ has $\sS$-length $n > 1$ for the remainder of the argument. We now fix an $\sS$-composition series for $s_f$:
\begin{align*}
& s_1 \rightlabel{i_{n-1}} x_2 \rightlabel{j_{n-1}} s_2 \rightlabel{k_{n-1}} \Sigma s_1 \\
& x_2 \rightlabel{i_{n-2}} x_3 \rightlabel{j_{n-2}} s_3 \rightlabel{k_{n-2}} \Sigma x_2 \\
& \vdots \\
& x_{n-1} \rightlabel{i_1} s_f \rightlabel{j_1} s_n \rightlabel{k_1} \Sigma x_{n-1}.
\end{align*}
For each $1 \leq p < n$ we have the following octahedral diagram, where $x_1 = s_1$, $x_n = s_f$, $\sigma_0 = \sigma$ and $c_0 = c_\sigma$. 
\begin{equation}\label{eq:octahedralseries}
\xymatrix{
                                & \Sigma^{-1} c_{p-1} \ar[d] \ar@{=}[r]       & \Sigma^{-1} c_{p-1} \ar[d]              & \\  
x_{n-p} \ar@{=}[d] \ar[r]^-{i_p} & x_{n-p+1} \ar[d]^-{\sigma_{p-1}} \ar[r]^-{j_p} & s_{n-p+1} \ar[d]^-{\theta_p} \ar[r]^-{k_p} & \Sigma x_{n-p} \ar@{=}[d] \\
x_{n-p} \ar[r]_-{\sigma_p}       & s_{af} \ar[d]^-{\tau_{p-1}} \ar[r]_-{\tau_p}   & c_p \ar[d] \ar[r]                       & \Sigma x_{n-p} \\
                               & c_{p-1} \ar@{=}[r]                              & c_{p-1}                                        &
}
\end{equation}

\Step{3} {\it The map $\sigma_{p} \colon x_{n-p} \to s_{af}$ is nonzero for each $1 \leq p < n$.}

\medskip
\noindent
By repeated use of the Octahedral Axiom in $\sD$, there is a triangle 
\[
\xymatrix{
x_{n-p} \ar[r]^-{i_1 i_2 \cdots i_p} & s_f \ar[r] & y_p \ar[r] & \Sigma x_{n-p},
}
\]
in which $y_p \in \extn{\sS}$ has $\sS$-length $p < n$. Now consider the following commutative diagram.
\[
\xymatrix{
x_{n-p} \ar[dr]^-{i_1 \cdots i_p}  \ar@/^2pc/[drr]^-{\alpha_f i_1 \cdots i_p} \ar@/_2pc/[ddr]_-{\sigma_p} &                                        &                  \\
                                                                                                      & s_f \ar[d]^-{\sigma} \ar[r]^-{\alpha_f} & c_f \ar[d]^-{r_1} \\
                                                                                                      & s_{af} \ar[r]_-{\alpha_{af}}             & c_{af}
}
\]
If $\alpha_f i_1 \cdots i_p = 0$ then we have a factorisation
\[
\xymatrix{
x_{n-p} \ar@{..>}[dr]_-{0} \ar[r]^-{i_1 \cdots i_p} & s_f \ar[d]_-{\alpha_f} \ar[r] & y_p \ar@{-->}[dl]^-{\exists\, \alpha_p} \ar[r] & \Sigma x_{n-p} \\
                                                  & c_f                          &                               &
}
\]
making $\alpha_p \colon y_p \to c_f$ into a right $\extn{\sS}$-approximation. Since $\alpha_f \colon s_f \to c_f$ is a minimal right $\extn{\sS}$-approximation it follows that $s_f$ is a direct summand of $y_p$ by Lemma~\ref{lem:basic-props}. Therefore, by \ref{lem:Dugas} (3), $s_f$ has $\sS$-length at most $p  < n$, contradicting our assumption on the $\sS$-length of $s_f$. Hence $\alpha_f i_1 \cdots i_p \neq 0$. Now by the injectivity of $\Hom_\sD(\extn{\sS},r_1)$ from Step 2, it follows that $r_1 \alpha_f i_1 \cdots i_p \neq 0$. Hence $\alpha_{af} \sigma_p \neq 0$ so that $\sigma_p \neq 0$, as claimed.

\Step{4} {\it The map $\theta_1 \colon s_n \to c_1$ is nonzero.}

\medskip
\noindent
We establish the stronger statement that $\tau_1 \sigma \neq 0$. If $\tau_1 \sigma = 0$ then we have the following factorisation:
\begin{equation}\label{eq:i'1}
\xymatrix{
                         & s_f \ar@{-->}[dl]_-{i_1'} \ar[d]^-{\sigma} \ar@{..>}[dr]^-{0} &            & \\
x_{n-1} \ar[r]_-{\sigma_1} & s_{af} \ar[r]_-{\tau_1}                                       & c_1 \ar[r] & \Sigma x_{n-1}
}
\end{equation}
We claim that $\alpha_f i_1 \colon x_{n-1} \to c_f$ is a right $\extn{\sS}$-approximation. From this it follows that $s_f$ is a direct summand of $x_{n-1}$ by Lemma~\ref{lem:basic-props} so that by \cite[Lemma 2.7]{Dugas}, $s_f$ has $\sS$-length $n-1$, contradicting our starting assumption. Hence $\tau_1 \sigma \neq 0$. 

We now establish the claim. Suppose $\phi \colon x \to c_f$ is a morphism with $x \in \extn{\sS}$. Then since $\alpha_f$ is a right $\extn{\sS}$-approximation there exists $\phi' \colon x \to s_f$ such that $\phi = \alpha_f \phi'$. Now,
\[
r_1 \phi = r_1 \alpha_f \phi' \stackrel{\eqref{eq:define-r-s}}{=} \alpha_{af} \sigma \phi' \stackrel{\eqref{eq:i'1}}{=} \alpha_{af} \sigma_1 i_1' \phi' \stackrel{\eqref{eq:octahedralseries}}{=} \alpha_{af} \sigma i_1 i_1' \phi' \stackrel{\eqref{eq:define-r-s}}{=} r_1 \alpha_f i_1 i_1' \phi',
\]
whence $r_1(\phi - \alpha_f i_1 i_1' \phi') = 0$. We therefore have the following factorisation.
\[
\xymatrix{
                      & x \ar@{-->}[dl]_-{\exists} \ar[d]^-{\phi - \alpha_f i_1 i_1' \phi'} \ar@/^3.5pc/@{..>}[dr]^-{0} &              & \\
\Sigma^{-1} c_a \ar[r] & c_f \ar[r]_-{r_1}                                                                      & c_{af} \ar[r] & c_a
}
\]
However, $c_a \in (\Sigma \sS)\orth$ (see Lemma~\ref{lem:cone-orth}) implies that $\phi = \alpha_f i_1 i_1' \phi'$ and $\alpha_f i_1 \colon x_{n-1} \to c_f$ is a right $\extn{\sS}$-approximation, as claimed.

\Step{5} {\it The map $\theta_{n-1} \colon s_2 \to c_{n-1}$ is nonzero.}

\medskip
\noindent
We again show the stronger statement that $\tau_{n-1} \sigma_{n-2} \neq 0$. If $\tau_{n-1} \sigma_{n-2} = 0$ then we have the following factorisation.
\begin{equation}\label{eq:i-n-1-tilde}
\xymatrix{
                         & x_2 \ar@{-->}[dl]_-{\exists \, \tilde{i}_{n-1}} \ar[d]^-{\sigma_{n-2}} \ar@/^1pc/@{..>}[dr]^-{0} &               &  \\ 
s_1 \ar[r]_-{\sigma_{n-1}} & s_{af} \ar[r]_-{\tau_{n-1}}                                                            & c_{n-1} \ar[r] & \Sigma s_1
}
\end{equation}
Now $0 \neq \sigma_{n-1} = \sigma_{n-2} i_{n-1} = \sigma_{n-1} \tilde{i}_{n-1} i_{n-1}$ shows that $\tilde{i}_{n-1} i_{n-1} \neq 0$ and is thus an isomorphism. This means that $\tilde{i}_{n-1}$ is a split epimorhism (and $i_{n-1}$ is a split monomorphism). Hence, we can replace the triangle $\trilabels{s_1}{x_2}{s_2}{i_{n-1}}{j_{n-1}}{k_{n-1}}$ occurring in the $\sS$-composition series for $s_f$ with the triangle
\[
\trilabels{s_2}{x_2}{s_1}{\tilde{j}_{n-1}}{\tilde{i}_{n-1}}{0}.
\]
Applying the Octahedral Axiom in $\sD$ to the composition $\sigma_{n-2} \tilde{j}_{n-1}$ we get the following diagram.
\begin{equation}\label{eq:sigma-j-tilde-oct}
\xymatrix{
                                & \Sigma^{-1} c_{n-2} \ar[d] \ar@{=}[r]       & \Sigma^{-1} c_{n-2} \ar[d]              & \\  
s_2 \ar@{=}[d] \ar[r]^-{\tilde{j}_{n-1}} & x_2 \ar[d]^-{\sigma_{n-2}} \ar[r]^-{\tilde{i}_{n-1}} & s_1 \ar[d] \ar[r]^-{0} & \Sigma s_2 \ar@{=}[d] \\
s_2 \ar[r]_-{\tilde{\sigma}_{n-1}}       & s_{af} \ar[d]^-{\tau_{n-2}} \ar[r]_-{\tilde{\tau}_{n-1}}   & \tilde{c}_{n-1} \ar[d] \ar[r]                       & \Sigma s_2 \\
                               & c_{n-2} \ar@{=}[r]                              & c_{n-2}                                        &
}
\end{equation}
By the same argument as Step 3, the map $\tilde{\sigma}_{n-1} \neq 0$. But now 
\[
\tilde{\sigma}_{n-1} \stackrel{\eqref{eq:sigma-j-tilde-oct}}{=} \sigma_{n-2} \tilde{j}_{n-1} \stackrel{\eqref{eq:i-n-1-tilde}}{=} \sigma_{n-1} \tilde{i}_{n-1} \tilde{j}_{n-1} = 0, 
\]
giving a contradiction. Hence $\tau_{n-1} \sigma_{n-2} \neq 0$ and $\theta_{n-1} \neq 0$ as claimed.

\Step{6} {\it For $1 < p < n-1$, the map $\theta_p \colon s_{n-p+1} \to c_p$ is nonzero.}

\medskip
\noindent
If $\theta_p = 0$ then $k_p = 0$ and the triangle $\trilabels{x_{n-p}}{x_{n-p+1}}{s_{n-p+1}}{i_p}{j_p}{k_p}$ occurring in the $\sS$-composition series of $s_f$ is split. Therefore $j_p$ is a split epimorphism with right inverse $\tilde{j}_p$, say. Consider the corresponding split triangle,
\[
\trilabels{s_{n-p+1}}{x_{n-p+1}}{x_{n-p}}{\tilde{j}_p}{\tilde{i}_p}{0},
\]
and the octahedral diagram arising from the composition $\sigma_{p-1} \tilde{j}_p$.
\[
\xymatrix{
                                & \Sigma^{-1} c_{p-1} \ar[d] \ar@{=}[r]       & \Sigma^{-1} c_{p-1} \ar[d]              & \\  
s_{n-p+1} \ar@{=}[d] \ar[r]^-{\tilde{j}_p} & x_{n-p+1} \ar[d]^-{\sigma_{p-1}} \ar[r]^-{\tilde{i}_p} & x_{n-p} \ar[d] \ar[r]^-{0} & \Sigma s_{n-p+1} \ar@{=}[d] \\
s_{n-p+1} \ar[r]_-{\tilde{\sigma}_p}       & s_{af} \ar[d]^-{\tau_{p-1}} \ar[r]_-{\tilde{\tau}_p}   & \tilde{c}_p \ar[d] \ar[r]                       & \Sigma s_{n-p+1} \\
                               & c_{p-1} \ar@{=}[r]                              & c_{p-1}                                        &
}
\]
By the argument of Step 3, we see again that $\tilde{\sigma}_p = \sigma_{p-1} \tilde{j}_p \neq 0$. Now, $\tau_p \tilde{\sigma}_p = \tau_p \sigma_{p-1} \tilde{j}_p \stackrel{\eqref{eq:octahedralseries}}{=} \theta_p j_p \tilde{j}_p = \theta_p = 0$, so that we have the following factorisation.
\[
\xymatrix{
                         & s_{n-p+1} \ar@{-->}[dl]_-{\exists \psi} \ar[d]^-{\tilde{\sigma}_p} \ar@/^1pc/@{..>}[dr]^-{0} &            & \\
x_{n-p} \ar[r]_-{\sigma_p} & s_{af} \ar[r]_-{\tau_p}                                                                    & c_p \ar[r] & \Sigma x_{n-p}
}
\]
Note that $\psi \neq 0$ because $\sigma_p \psi = \tilde{\sigma}_p \neq 0$. We claim that $\tilde{j}_p = i_p \psi$. We first show that this claim completes the argument before establishing the claim. Since $\tilde{j}_p$ is a right inverse for $j_p$ we have $1_{s_{n-p+1}} = j_p \tilde{j}_p = j_p i_p \psi = 0$, where the final equality follows from the composition of two consecutive morphisms in a triangle. This can only occur if $s_{n-p+1} = 0$, in which case $x_{n-p+1} \cong x_{n-p}$ has $\sS$-length at most $n-p$. In particular, this means $s_f$ has $\sS$-length strictly smaller than $n$, contradicting our starting assumption. Therefore, if $\tilde{j}_p = i_p \psi$ then we obtain a contradiction to the assumption that $\theta_p = 0$.

In order to establish the claim, we will need the following well-known lemma.

\begin{lemma} \label{lem:serre-factor}
Let $\sD$ be a triangulated category with Serre functor $\SSS \colon \sD \to \sD$. Let $z, w \in \sD$ and suppose we have a composition of morphisms $z \rightlabel{h} \SSS z \rightlabel{h'} w$ in which $h'$ is not a split monomorphism, and $h$ is the universal map $z \to \SSS z$. Then $h' h = 0$.
\end{lemma}

Observing that $\tilde{\sigma}_p = \sigma_p \psi \stackrel{\eqref{eq:octahedralseries}}{=} \sigma_{p-1} i_p \psi$ and $\tilde{\sigma}_p = \sigma_{p-1} \tilde{j}_p$, we see that $\sigma_{p-1}(\tilde{j}_p - i_p \psi) = 0$. Recall that $\sigma_{p-1} = \sigma i_1 \cdots i_{p-1}$ and consider the following diagram.
\[
\xymatrix{
s_{n-p+1} \ar[dr]^-{i_1 \cdots i_{p-1}(\tilde{j}_p - i_p \psi)}  \ar@{..>}@/_2pc/[ddr]_-{0} &                                        &                  \\
                                                                                                      & s_f \ar[d]^-{\sigma} \ar[r]^-{\alpha_f} & c_f \ar[d]^-{r_1} \\
                                                                                                      & s_{af} \ar[r]_-{\alpha_{af}}             & c_{af}
}
\]
In particular, this means that $r_1 \alpha_f i_1 \cdots i_{p-1}(\tilde{j}_p - i_p \psi) = 0$, whence by the injectivity of $\Hom_\sD(\extn{\sS},r_1)$ shown in Step 2, we have $\alpha_f i_1 \cdots i_{p-1}(\tilde{j}_p - i_p \psi) = 0$. We therefore get the following factorisation.
\[
\xymatrix{
                      & s_{n-p+1} \ar@{-->}[dl]_-{\exists} \ar[d]^-{i_1 \cdots i_{p-1}(\tilde{j}_p - i_p \psi)}  &            & \\ 
\Sigma^{-1} z_f \ar[r] & s_f \ar[r]_-{\alpha_f}                                                                                         & c_f \ar[r] & z_f
}
\]
But since $z_f \in \sZ$, the morphism labelled $\exists$ must be zero, whence $i_1 \cdots i_{p-1}(\tilde{j}_p - i_p \psi) = 0$.

We now, therefore, have the factorisation below.
\[
\xymatrix{
                      & s_{n-p+1} \ar@{-->}[dl]_-{\exists} \ar[d]^-{i_2 \cdots i_{p-1}(\tilde{j}_p - i_p \psi)}  &            & \\ 
\Sigma^{-1} s_n \ar[r]_-{-\Sigma^{-1}k_1} & x_{n-1} \ar[r]_-{i_1}                                                                                         & s_f \ar[r] & s_n
}
\]
Now, if assumption (2) of Setup~\ref{blanket-setup} holds, the morphism labelled $\exists$ in this diagram is also zero, so that $i_2 \cdots i_{p-1}(\tilde{j}_p - i_p \psi) = 0$. Otherwise, if assumption (1) holds, then $\sS$ is an $\SSS_{-1}$-subcategory, so that
\[
\Hom_\sD(s_{n-p+1},\Sigma^{-1} s_n) = \Hom_\sD(s_{n-p+1}, \SSS s_n) \simeq D\Hom_\sD(s_n,s_{n-p+1}). 
\]
If $s_n \not\simeq s_{n+p-1}$ then the morphism labelled $\exists$ is zero and we conclude that $i_2 \cdots i_{p-1}(\tilde{j}_p - i_p \psi) = 0$. If $s_n \simeq s_{n-p+1}$ and the morphism labelled $\exists$ is nonzero then it is, up to scalar, the universal morphism $s_n \to \SSS s_n$. Since $-\Sigma^{-1} k_1$ is not a split monomorphism, for otherwise $s_f$ would be a direct summand of $x_{n-1}$ and by \cite[Lemma 2.7]{Dugas} be of $\sS$-length strictly smaller than $n$, we can invoke Lemma~\ref{lem:serre-factor} to conclude that $i_2 \cdots i_{p-1}(\tilde{j}_p - i_p \psi) = 0$ also in this case. Repeating this argument a further $p-2$ times, we obtain that $\tilde{j}_p - i_p \psi = 0$, which is what we claimed, concluding Step 6.

\medskip
\noindent{\bf Conclusion.}
Since $\sigma_{n-1} \neq 0$, Lemma~\ref{lem:Dugas} implies that $c_{n-1} \in \extn{\sS}$. Using Lemma~\ref{lem:Dugas} again and the fact that $\theta_p \neq 0$ for $1 \leq p < n$, we obtain that $c_{p-1} \in \extn{\sS}$. In particular $c_0 = c_\sigma \in \extn{\sS}$, which is what we aimed to show. 
\end{proof}

\section{Calabi-Yau reduction} \label{sec:cy-reduction}

For a collection of objects $\sX$ of $\sD$ and $w \geq 1$, we define the following perpendicular categories: 
\begin{align*}
\sX\orthw    & \coloneqq \{d \in \sD \mid \Hom (\Sigma^i \sX, d) = 0 \text{ for } i= 0, \ldots, w\}, \text{ and} \\
{}\orthw \sX & \coloneqq \{d \in \sD \mid \Hom (d, \Sigma^i \sX) = 0 \text{ for } i= -w, \ldots, 0\}.
\end{align*}
Recall the definition of $\SSS_w$-subcategory from Definition~\ref{def:Sw-subcat}.
In this section, we will consider the following set up.

\begin{setup}\label{setup:reduction}
Let $w \geq 1$. Let $\sS$ be a $w$-orthogonal collection and $\sZ$ be a subcategory of $\sD$ satisfying the following conditions:
\begin{enumerate}
\item $\sS$ is an $\SSS_{-w}$-subcategory and $\extn{\sS}$ is functorially finite; and,
\item $\sZ = \sS\orthw$. 
\end{enumerate}
\end{setup}

The following lemma is a routine check.

\begin{lemma} \label{lem:Sw-subcat}
Let $\sS$ and $\sZ$ be as in Setup~\ref{setup:reduction}. Then $\sZ = \sS\orthw = {}\orthw \sS$ is also an $\SSS_{-w}$-subcategory.
\end{lemma}

We will now check that this set up satisfies the conditions in Setup~\ref{blanket-setup} and the hypotheses of Theorems~\ref{thm:pretriangulated} and \ref{thm:triangulated}. 

\begin{lemma}\label{lem:setups}
Let $\sS$ and $\sZ$ be as in Setup~\ref{setup:reduction}. Then the following conditions hold:
\begin{enumerate}
\item $\sS$ is either an $\SSS_{-1}$-subcategory of $\sD$ or $\Hom (\Sigma \sS, \sS) = 0$;
\item $\sZ$ is closed under extensions and direct summands;  
\item $(\sZ, \sZ)$ is an $\sS$-mutation pair;
\item the cones in $\sD$ of maps in $\sZ$ lie in $\extn{\sS} * \sZ$; and
\item the cocones in $\sD$ of maps in $\sZ$ lie in $\sZ * \extn{\sS}$.
\end{enumerate}
\end{lemma}

\begin{proof}
It is clear that Setup~\ref{setup:reduction} satisfies $(1)$ and $(2)$. 

In order to show $(3)$, let 
\[
\sZ_1 \coloneqq {}\orth \sS\orth \cap {}\orth(\Sigma^{-1} \sS) \cap \Sigma^{-1} (\extn{\sS}*\sZ).
\]
We need to show that $\sZ = \sZ_1$. 
Let $z' \in \sZ_1$. Then we have a triangle in $\sD$ of the form $\Sigma^{-1} s \to z' \to \Sigma^{-1} z \to s$, where $s \in \extn{\sS}$ and $z \in \sZ$. By applying $\Hom (\sS, -)$ to this triangle, we get that $\Hom (\Sigma^i \sS, z') = 0$, for $0 \leq i \leq w-2$, since $z \in \sS\orthw$ and $\sS$ is $w$-orthogonal. On the other hand, using the fact that $\sS$ is an $\SSS_{-w}$-subcategory, we have $\Hom (\Sigma^w \sS, z') \simeq \Hom (\SSS^{-1} \sS, z') \simeq D \Hom (z', \sS) = 0$, since $z' \in {}\orth \sS$, and $\Hom (\Sigma^{w-1} \sS, z') \simeq D \Hom (z', \Sigma^{-1} \sS) = 0$, as $z' \in {}\orth(\Sigma^{-1} \sS)$. Therefore, $z' \in \sS\orthw = \sZ$. 

Conversely, let $z \in \sZ$. Since $\extn{\sS}$ is functorially finite, we can consider the triangle $\xymatrix{s \ar[r]^f & \Sigma z \ar[r] & u \ar[r] & \Sigma s}$, where $f$ is a minimal right $\extn{\sS}$-approximation of $\Sigma z$. If we show that $u \in \sZ$, then $z \in \sZ_1$. We have $u \in \sS\orth$ by Lemma~\ref{lem:basic-props}(4), and $u \in (\Sigma^i \sS)\orth$, for $2 \leq i \leq w$, by applying $\Hom (\sS, -)$ to the triangle above and by using the fact that $z \in Z = \sS\orthw$ and $\sS$ is $w$-orthogonal. Finally, $u \in (\Sigma \sS)\orth$ follows from Lemma~\ref{lem:torsion-pair}(1). Hence, $u \in \sS\orthw = \sZ$, and so $\sZ = \sZ_1$. The proof that $\sZ = {}\orth \sS\orth \cap (\Sigma \sS)\orth \cap \Sigma (\sZ*\extn{\sS})$ is similar. This concludes the proof that $(\sZ, \sZ)$ is an $\sS$-mutation pair. 

Finally, we prove $(4)$. Let $f\colon x \to y$ be a map in $\sZ$ and consider the triangles in $\sD$:
\[
\xymatrix{x \ar[r]^f & y \ar[r] & c_f \ar[r] & \Sigma x} \text{ and }  \xymatrix{s_f \ar[r]^{\alpha_f} & c_f \ar[r]^{\beta_f} & z_f \ar[r]^{\gamma_f} & \Sigma s_f}.
\]
We want to show that $c_f \in \extn{\sS}*\sZ$. Given the right-hand triangle, it is enough to show that $z_f \in \sZ$. By Lemma~\ref{lem:basic-props}(4), $z_f \in \sS\orth$. Applying $\Hom (\sS, -)$ to the triangles above, we get $c_f \in (\Sigma^i \sS)\orth$, for $1 \leq i \leq w$, which implies that $z_f \in (\Sigma^j \sS)\orth$, for $2 \leq j \leq w$. Again, by Lemma~\ref{lem:torsion-pair}(1), $\Hom (\sS, \alpha_f)$ is an isomorphism, implying that $z_f \in (\Sigma \sS)\orth$. Therefore $z_f \in \sS\orthw = \sZ$, which finishes the proof of $(4)$. The proof of statement $(5)$ is dual. 
\end{proof}

In light of Lemma~\ref{lem:setups}, $\sZ$ has the structure of a triangulated category given in Theorems~\ref{thm:pretriangulated} and \ref{thm:triangulated}. Since there are two triangulated structures to consider, that in $\sD$ and that in $\sZ$, it is useful to set up some notation for that in $\sZ$ to distinguish between them.

\begin{notation} \label{notation}
Let $\sX$ and $\sY$ be subcategories of $\sZ$. We define 
\[
\sX \star \sY \coloneqq \{ z \in \sZ \mid 
\text{there exists a triangle } x \to z \to y \to x\extn{1} \text{ with } x \in \sX \text{ and } y \in \sY\}.
\]
We denote the extension closure of $\sX$ with respect to the triangulated structure in $\sZ$ by $\extnZ{\sX}$.
The usual notation $\sX * \sY$ and $\extn{\sX}$ keep their usual meanings \emph{in} $\sD$.
\end{notation}

\begin{lemma} \label{lem:R-filtration}
Let $\sS$ be a $w$-orthogonal collection and $\sZ$ be the subcategory of $\sD$ satisfying the hypotheses of Setup~\ref{setup:reduction}. Suppose $\sS \subseteq \sT$ for some $w$-simple-minded system $\sT$ in $\sD$. Write $\sR = \sT \setminus \sS$. Then $\extn{\sT} \cap \sZ = \extnZ{\sR}$.
\end{lemma}

\begin{proof}
First note that $w$-orthogonality of $\sT$ gives $\sR \subseteq \sZ = \sS\orthw$.

We now show the inclusion $\extnZ{\sR} \subseteq \extn{\sT} \cap \sZ$. Let $r \in \extnZ{\sR}$. Write $\extnZ{\sR}_n$ for the set of objects of $\extnZ{\sR}$ of $\sR$-length (in $\sZ$) at most $n$; see Definition~\ref{def:length}. We proceed by induction on the $\sR$-length of $r$. If $r$ has  $\sR$-length one, then $r \in \sR \subseteq \sT$ and there is nothing to show. 
Let $r \in \extnZ{\sR}_n$ for $n > 1$. Then there exists a triangle, $r_1 \rightlabel{f} r \too r' \too r_1 \extn{1}$,  in $\sZ$ with $r_1 \in \sR$, $r' \in \extnZ{\sR}_{n-1}$ and $f \neq 0$.
By Definition~\ref{def:cones-in-Z}, such a triangle comes from a diagram of the following form.
\[
\xymatrix@!R=8px{
                  &                       & s_f \ar[d] \ar[r]      & s_{r_1} \ar[d]  \\
r_1 \ar[r]^f & r \ar[r] \ar[dr]  & c_f \ar[r] \ar[d]       & \Sigma r_1 \ar[d]  \\
                 &                        & r' \ar[r] \ar[d]         & r_1 \extn{1} \ar[d]  \\
                 &                         & \Sigma s_f \ar[r]  & \Sigma s_{r_1}
}
\]
By induction $r' \in \extn{\sT} \cap \sZ$ so that $c_f \in \extn{\sT}$. Hence $r \in \extn{\sT} \cap \sZ$, as required.

For the converse, suppose $z \in \extn{\sT} \cap \sZ$. We proceed by induction on the $\sT$-length of $z$. If $z$ has $\sT$-length of one, then $z \in \sT \cap \sZ = \sR$ and the claim holds. Suppose $z \in (\sT)_n \cap \sZ$ for some $n > 1$. Since $z \in \sZ \subseteq \sS\orth$ and $(\sT)_n = \sT * (\sT)_n$, there is an object $r \in \sR = \sT \cap \sZ$ and a nonzero morphism $f \colon r \to z$. Consider the cone of $f$ constructed in Definition~\ref{def:cones-in-Z}.
\[
\xymatrix@!R=8px{
              &                       & s_f \ar[d] \ar[r]      & s_r \ar[d]  \\
r \ar[r]^f  & z \ar[r] \ar[dr]  & c_f \ar[r] \ar[d]     & \Sigma r \ar[d]  \\
              &                        & z_f \ar[r] \ar[d]     & r \extn{1} \ar[d]  \\
              &                        & \Sigma s_f \ar[r]  & \Sigma s_r
}
\]
By Lemma~\ref{lem:Dugas}(2), $c_f \in (\sT)_{n-1}$. Observe now that the left-hand vertical triangle is the right $\extn{\sS}$-approximation triangle occurring in Theorem~\ref{thm:funct-finite}. It follows that $z_f \in (\sT)_m$ for some $m < n$. By construction, $z_f \in \sZ$, so by induction $z_f \in \extnZ{\sR}$. The triangle $r \rightlabel{f} z \too z_f \too r \extn{1}$ in $\sZ$ now shows that $z \in \extnZ{\sR}$, completing the proof.
\end{proof}

We are now ready to prove the main theorem of this section which provides an inductive technique for constructing $w$-simple-minded systems.

\begin{theorem}[Reduction for simple-minded systems] \label{thm:bijectionsms}
Let $\sS$ be a $w$-orthogonal collection and $\sZ$ be the subcategory of $\sD$ satisfying the hypotheses of Setup~\ref{setup:reduction}. Then there is bijection,
\[
\{\text{$w$-simple-minded systems in $\sD$ containing $\sS$}\} \stackrel{1-1}{\longleftrightarrow}
\{\text{$w$-simple-minded systems in $\sZ$}\}.
\]
\end{theorem}

\begin{proof}
Let $\sT$ be a $w$-simple-minded system in $\sD$ such that $\sS \subseteq \sT$. We will show that $\sR \coloneqq \sT \setminus \sS$ is a $w$-simple-minded system in $\sZ$.
Recall from Lemma~\ref{lem:R-filtration} that $\sR \subseteq \sZ = \sS\orthw$. 

We first show that $\sR$ is $w$-orthogonal in $\sZ$. Let $r_1, r_2 \in \sR$. Clearly, $\Hom_\sD (r_1, r_2) = \delta_{r_1, r_2} \kk$, so it remains to show that, if $w \geq 2$, then $\Hom_{\sD} (r_1 \extn{k}, r_2) = 0$, for $1 \leq k \leq w-1$. By Definition~\ref{def:cones-in-Z}, for $1 \leq k \leq w-1$, we have triangles in $\sD$ of the form,
\[
\tri{s_{r_1\extn{k-1}}}{\Sigma r_1 \extn{k-1}}{r_1 \extn{k}}.
\]
Since $\sR \subseteq \sS\orthw$, applying the functor $\Hom_{\sD} (-,r_2)$ to the triangles above gives
\[
\Hom_\sD (r_1 \extn{k}, r_2) \simeq \Hom_\sD (\Sigma r_1 \extn{k-1}, r_2) \simeq \Hom_\sD (\Sigma^2 r_1 \extn{k-2}, r_2) \simeq \cdots \simeq \Hom_\sD (\Sigma^k r_1, r_2),
\]
for each $1 \leq k \leq w-1$. Since $\sT$ is $w$-orthogonal in $\sD$, we have $\Hom_\sD (\Sigma^k r_1, r_2) = 0$, for $1 \leq k \leq w-1$, and the $w$-orthogonality of $\sR$ in $\sZ$ follows. 

Next, we will prove that $\sR$ is a right $w$-Riedtmann configuration in $\sZ$. Suppose $z \in \sZ$ is such that $\Hom_\sD (z \extn{k}, \sR) = 0$, for $0 \leq k \leq w-1$. Following the dimension shifting argument as above, we have $\Hom_\sD (\Sigma^k z, \sR) \simeq \Hom_\sD (z \extn{k}, \sR) = 0$. On the other hand, we have $\Hom_\sD (\Sigma^k z, \sS) = 0$, for $0 \leq k \leq w-1$, by definition of $\sZ$. Hence, $\Hom_\sD (\Sigma^k z, \sT) = 0$, for $0 \leq k \leq w-1$ because $\sT = \sR \cup \sS$. 
By Proposition~\ref{prop:Riedtmann}, $\sT$ is a right $w$-Riedtmann configuration. It then follows that $z = 0$, showing that $\sR$ is a right $w$-Riedtmann configuration in $\sZ$. 

Finally, we will show that $\extnZ{\sR}$ is covariantly finite in $\sZ$. Let $z \in \sZ$. Since $\sT$ is a $w$-simple-minded system, by Proposition~\ref{prop:Riedtmann}, $\extn{\sT}$ is covariantly finite in $\sD$. Take a left $\extn{\sT}$-approximation triangle for $z$ in $\sD$:
\[
\trilabels{x}{z}{t_z}{f}{g_1}{h_1},
\]
and note that $x \in {}\orth \sT$.
We first claim that $x \in \sZ$. Applying the functor $\Hom(-,\sS)$ to the left $\extn{\sT}$-approximation triangle above yields a long exact sequence for $0 \leq i \leq w-1$,
\[
(z, \Sigma^{-i-1} \sS) \too (x, \Sigma^{-i-1} \sS) \too (t_z, \Sigma^{-i} \sS) \too (z, \Sigma^{-i} \sS).
\]
For $0 \leq i \leq w$, $\Hom(z, \Sigma^{-i} \sS) = 0$ because $z \in \sZ = {}\orthw \sS$. Thus the left-hand and right-hand terms vanish for $0 \leq i \leq w-1$.
Since $t_z \in \extn{\sT}$, we have $\Hom(t_z, \Sigma^{-i} \sS) = 0$ for $1 \leq i \leq w-1$ by $w$-orthogonality of $\sT$.
Therefore, $\Hom(x, \Sigma^{-i} \sS) = 0$ for $i = 0$ since $x\in {}\orth \sT \subseteq {}\orth \sS$ and $2 \leq i \leq w$.
To see that $\Hom(x, \Sigma^{-1} \sS) = 0$, apply Lemma~\ref{lem:torsion-pair}(1) to see that $\Hom(t_z,\sT) \too \Hom(z,\sT)$ is an isomorphism and observe that $\Hom(z,\sS) = 0$ because $z \in \sZ$. This shows that $x \in \sZ$.

Since $f\colon x \to z$ is a morphism in $\sZ$, we use Definition~\ref{def:cones-in-Z} to construct the cone in $\sZ$.
\[
\xymatrix@!R=8px{
              &                                            & s_f \ar[d] \ar[r]              & s_x \ar[d]  \\
x \ar[r]^f  & z \ar[r]^-{g_1} \ar[dr]_-{g}   & t_z \ar[r]^-{h_1} \ar[d]  & \Sigma x \ar[d]  \\
              &                                            & z_f \ar[r]^-{h} \ar[d]       & x \extn{1} \ar[d]  \\
              &                                            & \Sigma s_f \ar[r]           & \Sigma s_x
}
\]
This produces a triangle $x \rightlabel{f} z \rightlabel{g} z_f \rightlabel{h} x \extn{1}$ in $\sZ$.
Since $t_z \in \extn{\sT}$, one observes that the left-hand vertical triangle in the diagram above is the approximation triangle from Theorem~\ref{thm:funct-finite}. It follows that $z_f \in \extn{\sT} \cap \sZ$. By Lemma~\ref{lem:R-filtration}, we therefore have $z_f \in \extnZ{\sR}$.
To see that $g \colon z \to z_f$ is a left $\extnZ{\sR}$-approximation in $\sZ$, we need to see that $x \in {}\orth \sR$. However, since $\sR \subseteq \sT$ and $x \in {}\orth \sT$, this is clear.

In conclusion, we have shown that $\sR$ is a right $w$-Riedtmann configuration in $\sZ$ such that $\extnZ{\sR}$ is covariantly finite. By Proposition~\ref{prop:Riedtmann}, $\sR$ is a $w$-simple-minded system in $\sZ$.

Conversely, let $\sR$ be a $w$-simple-minded system in $\sZ$. We will show that $\sT \coloneqq \sR \cup \sS$ is a $w$-simple-minded system in $\sD$. The fact that $\sT$ is $w$-orthogonal in $\sD$ follows from $w$-orthogonality of $\sS$ in $\sD$, $\sR \subseteq \sZ = {}\orthw \sS = \sS\orthw$, and $\Hom_\sD (\Sigma^k r, \sR) \simeq \Hom_\sD (r \extn{k}, \sR)$, for all $r \in \sR$ and $0 \leq k \leq w-1$, as seen in the dimension shifting argument in the first implication. 

We now show that $\sD = \extn{\sT} * \extn{\Sigma^{-1} \sT} * \cdots * \extn{\Sigma^{1-w} \sT}$.
By Lemma~\ref{lem:funct-finite}, the subcategory $\sX \coloneqq \extn{\sS} * \cdots * \extn{\Sigma^{-w} \sS}$ is functorially finite in $\sD$. Therefore, there is a torsion pair $({}\orth \sX,\sX)$ in $\sD$. By definition of $\sZ$, one sees that $\sZ = {}\orth\sX$.
For $d \in \sD$, let 
\begin{equation} \label{decomposition}
\tri{z}{d}{x}
\end{equation}
be a decomposition triangle with respect to this torsion pair. 
Since $\sR$ is a $w$-simple-minded system in $\sZ$, we have $\sZ = \extnZ{\sR} \star \extnZ{\sR\extn{-1}} \star \cdots \star \extnZ{\sR\extn{1-w}}$ (recall Notation~\ref{notation}).

\begin{claim} \label{claimA}
For $1 \leq i \leq w-1$, we have  $\extnZ{\sR\extn{-i}} \subseteq \extn{\Sigma^{-1} \sT} * \cdots * \extn{\Sigma^{-i} \sT}$.
\end{claim} 

\begin{proof}[Proof of claim] \renewcommand{\qedsymbol}{} 
For $i = 1$, let $z \in \extnZ{\sR}$. By Definition~\ref{def:shift-in-Z}, we have a triangle
\[
\Sigma^{-1} s^z \to z \extn{-1} \to \Sigma^{-1} z \to s^z,
\]
from which it follows that $z \extn{-1} \in \extn{\Sigma^{-1} \sS} * \Sigma^{-1} \extnZ{\sR} \subseteq \extn{\Sigma^{-1} \sS} * \Sigma^{-1} \extn{\sT} \subseteq \extn{\Sigma^{-1} \sT}$, where the first inclusion is by Lemma~\ref{lem:R-filtration}.

Now suppose $i > 1$, and assume the claim holds for $i-1$. Let $z \in \extnZ{\sR}$. By induction, $z\extn{-i+1} \in \extnZ{\sR\extn{-i+1}} \subseteq \extn{\Sigma^{-1} \sT} * \cdots * \extn{\Sigma^{-i+1} \sT}$. Again considering the triangle from Definition~\ref{def:shift-in-Z},
\[
\Sigma^{-1} s^{z\extn{-i+1}} \to z \extn{-i} \to \Sigma^{-1} z \extn{-i+1} \to s^{z\extn{-i+1}},
\]
shows that $z \extn{-i} \in \extn{\Sigma^{-1} \sS} * \Sigma^{-1} \extnZ{\sR\extn{-i+1}} \subseteq \extn{\Sigma^{-1} \sS} * \Sigma^{-1}(\extn{\Sigma^{-1} \sT} * \cdots * \extn{\Sigma^{-i+1} \sT}) \subseteq \extn{\Sigma^{-1} \sT} * \cdots * \extn{\Sigma^{-i} \sT}$, establishing the claim.
\end{proof}

\begin{claim} \label{claimB}
For $0 \leq i \leq w-1$, we have $\extnZ{\sR} \star \cdots \star \extnZ{\sR\extn{-i}} \subseteq \extn{\sT}* \cdots * \extn{\Sigma^{-i} \sT}$.
\end{claim}

\begin{proof}[Proof of claim] \renewcommand{\qedsymbol}{} 
The case $i = 0$ is Lemma~\ref{lem:R-filtration}. Suppose $i \geq 1$ and the claim holds for $i-1$. 
Let $z \in \extnZ{\sR} \star \extnZ{\sR\extn{-1}} \star \cdots \star \extnZ{\sR\extn{-i}}$. Therefore, there is a triangle
\[
x \to z \to r\extn{-i} \to x\extn{1}
\]
with $x \in \extnZ{\sR} \star \extnZ{\sR\extn{-1}} \star \cdots \star \extnZ{\sR\extn{-i+1}}$ and $r \extn{-i} \in \extnZ{\sR\extn{-i}}$.
By induction $x \in \extn{\sT} * \cdots * \extn{\Sigma^{-i+1} \sT}$ and by Claim~\ref{claimA}, $r \extn{-i} \in \extn{\Sigma^{-1} \sT} * \cdots * \extn{\Sigma^{-i} \sT}$.
By Lemma~\ref{lem:reverseorderextension}, it follows that $z \in \extn{\sT} * \cdots * \extn{\Sigma^{-i} \sT}$, giving the claim.
\end{proof}

Returning to the decomposition triangle \eqref{decomposition}, we see that by Claim~\ref{claimB} the object $z \in \extn{\sT} * \extn{\Sigma^{-1} \sT} * \cdots * \extn{\Sigma^{1-w} \sT}$. Now applying Lemma~\ref{lem:reverseorderextension} we obtain that $d \in \extn{\sT} * \extn{\Sigma^{-1} \sT} * \cdots * \extn{\Sigma^{1-w} \sT}$. Hence $\sT$ is a $w$-simple-minded system in $\sD$.

Clearly the maps $\sT \mapsto \sT \setminus \sS$ and $\sR \mapsto \sR \cup \sS$ are mutually inverse bijections, finishing the proof.   
\end{proof}

Finally, we verify that the property of satisfying Serre duality is preserved by simple-minded reduction. In particular, this means that the type of the category is preserved under our reduction procedure.

\begin{theorem} \label{thm:Serre}
Let $\sS$ be a $w$-orthogonal collection and $\sZ$ be a subcategory of $\sD$ satisfying the hypotheses of Setup~\ref{setup:reduction}.
Suppose $\sD$ satisfies Serre duality with Serre functor $\SSS \colon \sD \to \sD$. Then $\sZ$ is a triangulated category with Serre functor $\overline{\SSS} := \SSS \Sigma^{w} \extn{-w}$.
In particular, if $\sD$ is $(-w)$-Calabi-Yau, so is $\sZ$.  
\end{theorem}

\begin{proof}
The subcategory $\sZ$ is triangulated by Theorems~\ref{thm:pretriangulated} and \ref{thm:triangulated}.
By Lemma~\ref{lem:Sw-subcat}, we have $\SSS \sZ = \Sigma^{-w} \sZ$. This implies that $\overline{\SSS} z \in \sZ$ for each $z \in \sZ$.

We shall show that $\overline{\SSS} = \Sigma^w \SSS \extn{-w} \colon \sZ \to \sZ$ is a right Serre functor. That is, for $x, y \in \sZ$ we need to check that there is a natural isomorphism 
$\Hom(x,y) \simeq D \Hom(y, \overline{\SSS} x)$.
Applying the construction of the shift functor $\extn{1}$ in $\sZ$ from Definition~\ref{def:shift-in-Z} iteratively yields the following triangles in $\sD$:
\[
\tri{s_i}{\Sigma x \extn{-i}}{x \extn{-i+1}},
\]
to which we apply the functor $\SSS \Sigma^{i-1}$ to give:
\[
\SSS \Sigma^{i-1} s_i  \to \SSS \Sigma^{i} x \extn{-i} \to \SSS \Sigma^{i -1} x\extn{-i+1} \rightarrow \SSS \Sigma^{i} s_i.
\]
We now apply the functor $\Hom (y, -)$ to these triangles to give the long exact sequences.
\[
\cdots \to (y, \SSS \Sigma^{i-1} s_i) \to (y, \SSS \Sigma^i x \extn{-i}) \to (y, \SSS \Sigma^{i-1} x \extn{-i+1}) \to (y, \SSS \Sigma^i s_i) \to \cdots
\]
Using the fact that $\sZ = {}\orthw \sS = \sS\orthw$ and Serre duality we have 
$\Hom (y, \SSS \Sigma^{i-1} s_i) \simeq D \Hom (\Sigma^{i-1} s_i, y) = 0$ and $\Hom (y, \SSS \Sigma^{i} s_i) \simeq D \Hom (\Sigma^i s_i, y) = 0$, for each $1 \leq i \leq w$, so that
\[
\Hom(y, \SSS \Sigma^i x \extn{-i}) \simeq \Hom(y,\SSS \Sigma^{i-1} x \extn{-i+1}) \text{ for } 1 \leq i \leq w.
\]
Putting these together, we obtain the desired isomorphism:
\[
D \Hom(y, \overline{\SSS} x ) \simeq D \Hom (y, \SSS \Sigma^w x \extn{-w})  \simeq D \Hom (y, \SSS x) \simeq  \Hom (x,y).
\]
A similar argument shows that $\overline{\SSS}^{-1} \coloneqq \Sigma^{-w} \SSS^{-1} \extn{w}$ is a left Serre functor for $\sZ$. By \cite[Lemma I.1.5]{RvdB}, $\overline{\SSS} \colon \sZ \to \sZ$ is a Serre functor. 
For the final statement, observe that if $\sD$ is $(-w)$-Calabi-Yau then $\Sigma^w \SSS \simeq \id_\sD$ so that the final isomorphism in $\sZ$ becomes
$D \Hom(y, x \extn{-w}) \simeq \Hom(y,x)$ and $\sZ$ is also $(-w)$-Calabi-Yau.
\end{proof}

\section{Examples} \label{sec:examples}

In this section we briefly discuss two simple examples. The first example illustrates the construction of Theorem~\ref{thm:triangulated} in an orbit category of the derived category of the path algebra of a Dynkin type $A$ quiver. The second example explains in a similar context why the Iyama-Yoshino subfactor construction from \cite{IY} does not work in our context.

\begin{example}
Let $\sD = \Db(\kk A_5)/\Sigma^3\tau$, where $A_5$ is the linearly oriented Dynkin quiver of type $A_5$ and $\tau$ is the Auslander--Reiten translate in $\Db(\kk A_5)$. The Auslander--Reiten quiver of $\sD$ is indicated in Figure~\ref{fig:example}.
Let $\sS = \{s_1,s_2\}$ be as indicated in Figure~\ref{fig:example}, since we are in finite type, $\extn{\sS}$ is clearly functorially finite. Let $\sZ = {}^{\perp_2} \sS^{\perp_2}$. By Theorem~\ref{thm:triangulated}, $\sZ$ is a triangulated category; moreover, we have 
\[
\sZ \simeq \Db(\kk A_2)/\Sigma^3\tau \oplus \Db(\kk A_1)/\Sigma^3\tau .
\]
The component $\Db(\kk A_2)/\Sigma^3\tau$ is indicated in blue in Figure~\ref{fig:example} and $\Db(\kk A_1)/\Sigma^3\tau$ in green.

Let $x$ and $y$ be the indecomposable objects indicated in Figure~\ref{fig:example}. Up to scalars there is one nonzero map $f \colon x \to y$. The cone of this map, $c_f \notin \sZ$ is indicated. Our construction gives the cone $z_f \in \sZ$. We see also that $\Sigma x \notin \sZ$, but $x\extn{1} \in \sZ$, giving us the desired triangle $x \rightlabel{f} y \too z_f \too x \extn{1}$ in $\sZ$.

Now observe that $\sR = \{y, x \extn{1}, t\}$ as a $2$-simple-minded system in $\sZ$. By Theorem~\ref{thm:bijectionsms}, $\sS \cup \sR = \{s_1, s_2, y, x \extn{1}, t\}$ is a $2$-simple-minded system in $\sD$.

\begin{figure}
\begin{tikzpicture}[scale=0.8]

\draw[gray,rounded corners] (3,4.75) -- (0.25,-0.75) -- (10.75,-0.75) -- (8, 4.75) -- cycle;

\fill[red!10!white] (1,4) circle (2.5mm);
\draw[red,thick] (1,4) circle (2.5mm);

\fill[red!50!white,rounded corners] (1.7, 1.9) -- (2.1,1.7) -- (2.8, 3.1) -- (2.4, 3.3) -- cycle;
\draw[red,rounded corners,thick] (1.7, 1.9) -- (2.1,1.7) -- (2.8, 3.1) -- (2.4, 3.3) -- cycle;
\fill[red!50!white] (4,0) circle (2.5mm);
\draw[red,thick] (4,0) circle (2.5mm);

\fill[red!30!white,rounded corners] (4.7, 2.1) -- (5.1,2.3) -- (5.8, 0.9) -- (5.4, 0.7) -- cycle;
\draw[red,rounded corners,thick] (4.7, 2.1) -- (5.1,2.3) -- (5.8, 0.9) -- (5.4, 0.7) -- cycle;
\fill[red!30!white] (7,4) circle (2.5mm);
\draw[red,thick] (7,4) circle (2.5mm);

\fill[red!10!white,rounded corners] (6.7, 1.9) -- (7.1,1.7) -- (7.8, 3.1) -- (7.4, 3.3) -- cycle;
\draw[red,rounded corners,thick] (6.7, 1.9) -- (7.1,1.7) -- (7.8, 3.1) -- (7.4, 3.3) -- cycle;
\fill[red!10!white] (9,0) circle (2.5mm);
\draw[red,thick] (9,0) circle (2.5mm);

\fill[red!50!white,rounded corners] (9.7, 2.1) -- (10.1,2.3) -- (10.8, 0.9) -- (10.4, 0.7) -- cycle;
\draw[red,rounded corners,thick] (9.7, 2.1) -- (10.1,2.3) -- (10.8, 0.9) -- (10.4, 0.7) -- cycle;
\fill[red!50!white] (12,4) circle (2.5mm);
\draw[red,thick] (12,4) circle (2.5mm);

\fill[blue!30!white] (2,0) circle (2.5mm);
\draw[blue,thick] (2,0) circle (2.5mm);

\fill[blue!30!white,rounded corners] (4.5,2.6) -- (3.6,4.25) -- (5.4,4.25) -- cycle;
\draw[blue,rounded corners,thick] (4.5,2.6) -- (3.6,4.25) -- (5.4,4.25) -- cycle;

\fill[blue!30!white,rounded corners] (7.5,1.4) -- (6.6,-0.25) -- (8.4,-0.25) -- cycle;
\draw[blue,rounded corners,thick] (7.5,1.4) -- (6.6,-0.25) -- (8.4,-0.25) -- cycle;

\fill[blue!30!white] (10,4) circle (2.5mm);
\draw[blue,thick] (10,4) circle (2.5mm);

\fill[blue!30!white,rounded corners] (12.5,1.4) -- (11.6,-0.25) -- (13.4,-0.25) -- cycle;
\draw[blue,rounded corners,thick] (12.5,1.4) -- (11.6,-0.25) -- (13.4,-0.25) -- cycle;

\fill[green!30!white] (5,0) circle (2.5mm);
\draw[green,thick] (5,0) circle (2.5mm);

\fill[green!30!white] (8,4) circle (2.5mm);
\draw[green,thick] (8,4) circle (2.5mm);

\fill[green!30!white] (13,4) circle (2.5mm);
\draw[green,thick] (13,4) circle (2.5mm);

\node[below] (S1) at (2,1.8) {\tiny $s_1$};
\node[below] (S2) at (4,-0.2) {\tiny $s_2$};
\node[above] (S1') at (10,2.2) {\tiny $s_1$};
\node[above] (S2') at (12,4.2) {\tiny $s_2$}; 
\node[below] (x) at (8,-0.2) {\tiny $x$};
\node[above] (y1) at (10,4.2) {\tiny $y$};
\node[below] (y2) at (2,-0.2) {\tiny $y$};
\node[below] (cf1) at (2.6,0.9) {\tiny $c_f$};
\node[above] (cf2) at (10.6,3) {\tiny $c_f$};
\node[above] (x[1]1) at (11,4.2) {\tiny $\Sigma x$};
\node[below] (x[1]2) at (3,-0.2) {\tiny $\Sigma x$};
\node[above] (zf1) at (4, 4.2) {\tiny $z_f$};
\node[below] (zf2) at (12, -0.2) {\tiny $z_f$};
\node[below] (x<1>1) at (4.5, 2.8) {\tiny $x\extn{1}$};
\node[above] (x<1>2) at (12.5, 1.2) {\tiny $x\extn{1}$};
\node[below] (t1) at (5, -0.2) {\tiny $t$};
\node[above] (t2) at (13,4.2) {\tiny $t$};

\foreach \x in {1,2,...,13}
	\foreach \y in {0,2,...,4}
{	
\fill (\x, \y) circle (0.5mm);
}
\foreach \x in {1.5,2.5,...,13}
	\foreach \y in {1,3}
{	
\fill (\x, \y) circle (0.5mm);
}
\end{tikzpicture}
\caption{\label{fig:example}Auslander--Reiten quiver of $\sD$ with arrows omitted. A fundamental domain of the $(-2)$-CY orbit category $\sD$ is outlined in grey. The extension closure $\extn{\sS}$ is shaded dark red, $\extn{\Sigma \sS}$ mid-red, and $\extn{\Sigma^{-1} \sS}$ light red. $\sZ = {}^{\perp_2} \sS^{\perp_2}$ is shaded blue and green.}
\end{figure}
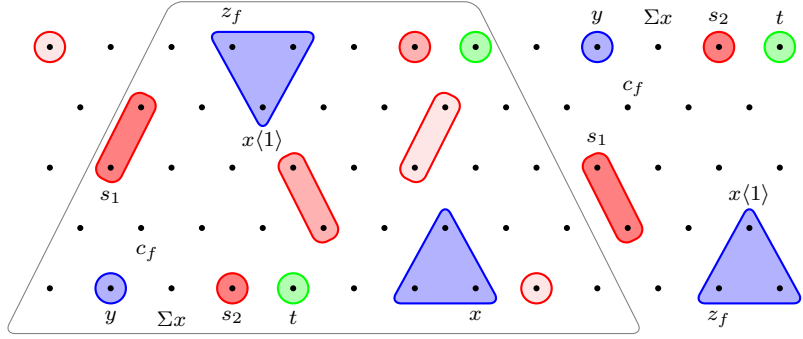
\end{example}

\begin{example} \label{subfactor}
For this example let $\sD = \Db(\kk A_3)/\Sigma^2\tau$, where $A_3$ is the linearly oriented Dynkin quiver of type $A_3$ and $\tau$ is the Auslander--Reiten translate in $\Db(\kk A_3)$. The category $\sD$ has nine isoclasses of indecomposable objects and the following AR quiver.
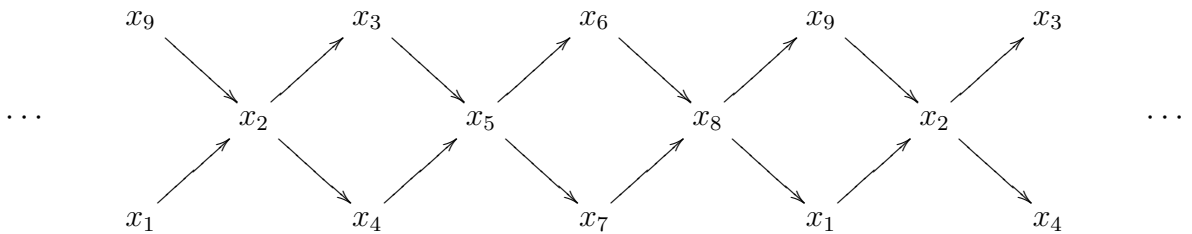
\[
\xymatrix{
           & x_9 \ar[dr] &                             & x_3 \ar[dr] &                             & x_6 \ar[dr] &                             & x_9 \ar[dr] &                             & x_3 & \\
\cdots &                 & x_2 \ar[dr] \ar[ur] &                   & x_5 \ar[dr] \ar[ur] &                   & x_8 \ar[dr] \ar[ur] &                  & x_2 \ar[dr] \ar[ur] &         & \cdots \\ 
           & x_1 \ar[ur] &                             & x_4 \ar[ur] &                             & x_7 \ar[ur] &                             & x_1 \ar[ur] &                             & x_4 &                 
}
\]
Let $\sS = \{x_1\}$, which is an orthogonal collection such that $\extn{\sS}$ is functorially finite. 
We will show that the triangulated structure on ${}\orth \sS\orth = \add\{x_4, x_5, x_7, x_9\}$ is not given by an Iyama-Yoshino subfactor construction.

Suppose ${}\orth \sS\orth = \sZ/[\sR]$ for some $\sR$-mutation pair $(\sZ,\sZ)$ in the sense of \cite{IY}. Note that $\sR \subset \sZ$ in this case. The objects $x_1, x_2, x_3, x_6$ and $x_8$ must be zero in $\sZ/[\sR]$ and $x_2, x_8 \notin \sR$ because we require nonzero morphisms $x_7 \to x_9$ and $x_9 \to x_4$.
Therefore $\sZ \subset \add\{x_1, x_3, x_4, x_5, x_6, x_7, x_9\}$. To get a triangle $x_7 \rightlabel{\bar{f}} x_9 \too x_4 \too x_7 \extn{1}$ in $\sZ/[\sR]$ we need a triangle in $\sD$ in which the first three terms lie in $\sZ$ whose image under the subfactor construction is this triangle. That is, we need a triangle of the form
\begin{equation} \label{candidate}
x_7 \oplus a \rightlabel{f} x_9 \oplus b \too x_4 \oplus c \too \Sigma(x_7 \oplus a),
\end{equation}
in which the summands $a, b, c \in \add\{x_1, x_3, x_6\}$ (these are the only summands which can be killed in the passage to $\sZ/[\sR]$).

Write $f = \begin{pmat} f_1 & f_2 \\ f_3 & f_4 \end{pmat}$. Clearly, $f_2 = 0$ for otherwise $a \in \add\{x_7,x_8, x_9\}$. Similarly, $f_3 = 0$ for otherwise $b \in \add\{x_7, x_8, x_9\}$. Therefore the triangle \eqref{candidate} is a direct sum of triangles,
\[
x_7 \oplus a \rightlabel{\begin{pmat} f_1 & 0 \\ 0 & f_4 \end{pmat}} x_9 \oplus b \too x_2 \oplus c \too \Sigma(x_7\oplus a),
\]
whose image under the subfactor construction is not the required triangle.
\end{example}


\end{document}